\documentclass{scrartcl}

\usepackage[T1]{fontenc}
\usepackage[latin1]{inputenc}

\usepackage{amssymb}
\usepackage{amsfonts}
\usepackage{amsmath}
\usepackage{wasysym}

\usepackage{mathtools}

\usepackage{epsfig}
\usepackage{epic}
\usepackage{eepic}
\usepackage{graphicx}
\usepackage{color}
\usepackage{longtable}

\usepackage{algorithm}
\usepackage{algorithmic}

\usepackage{float}

\usepackage{amsthm}

\usepackage{verbatim}

\newcommand{\id}{\operatorname{id}}

\newcommand{\Ord}{\operatorname{O}}

\newcommand{\diag}{\operatorname{diag}}

\newcommand{\range}{\mathcal{R}}

\newcommand{\Nn}{{\mathbb N}}

\newcommand{\Rr}{{\mathbb R}}

\newcommand{\eps}{\varepsilon}

\newcommand{\Hh}{\mathcal{H}}

\newcommand{\Kcal}{\mathcal{K}}

\newcommand{\resa}{\prescript{\rm a \! }{}{r}}

\newcommand{\mua}{\prescript{\rm a \! }{}{\mu}}

\theoremstyle{plain}
\newtheorem{Thm}{Theorem}[section]

\newtheorem{Cor}[Thm]{Corollary}
\newtheorem{Lem}[Thm]{Lemma}

\theoremstyle{remark}
\newtheorem{Rem}[Thm]{Remark}

\theoremstyle{definition}

\begin{document}

\title{Accelerated Landweber methods based on co-dilated orthogonal polynomials}
\author{Wolfgang Erb
\thanks{Institute of Mathematics, University of Lübeck,
Ratzeburger Allee 160, 23562 Lübeck, Germany. erb@math.uni-luebeck.de}}

\date{18.12.2012}

\maketitle
\vspace{-10mm}
\begin{abstract}
In this article, we introduce and study accelerated Landweber methods for linear ill-posed problems obtained by an alteration of the coefficients in the three-term recurrence relation
of the $\nu$-methods. The residual polynomials of the semi-iterative methods under consideration are linked to a family of co-dilated ultraspherical polynomials. This connection makes it possible
to increase the decay of the residual polynomials at the origin by means of a dilation parameter. This increased decay has advantages when solving linear ill-posed equations 
in which the spectrum of the involved operators is clustered at the origin. The convergence order of the new semi-iterative methods turns out to be the same as the convergence order of the original $\nu$-methods. The new algorithms are tested numerically and a simple adaptive scheme is developed in which an optimal dilation parameter is computed. 
\end{abstract}

\section{Introduction}

The aim of this article is to present and to investigate specific accelerated Landweber schemes that constitute an alternative to the well-known $\nu$-methods and which, depending on the given data, are able to display an improved performance. For the necessary notation, we give first a short summary about linear ill-posed problems, the Landweber iteration and semi-iterative methods. The theoretical background on ill-posed problems and their numerical solution is mainly taken from the monographs \cite{EnglHankeNeubauer} and \cite{Rieder}. Further introductions can be found in \cite{Groetsch,KaltenbacherNeubauerScherzer, Kirsch, Louis} and the references therein. 

In the Hilbert space setting of linear ill-posed problems, one considers a bounded linear operator $A: \Hh_1 \to \Hh_2$ between two Hilbert spaces $\Hh_1$ and $\Hh_2$ and the solutions of the linear equation
\begin{equation} \label{equation-illposedLGS}
A f = g,\qquad f \in \Hh_1, \; g \in \range(A).
\end{equation}
If the range $\range(A)$ of $A$ is not a closed subspace of $\Hh_2$, the solution $f$ of \eqref{equation-illposedLGS} does not depend continuously on the initial data $g$ and the linear system
\eqref{equation-illposedLGS} is called ill-posed. To solve such ill-posed equations, Landweber suggested in \cite{Landweber1951} the iterative scheme 
\begin{equation} \label{equation-Landeber}
f_{n+1} = f_n + 2 \omega  A^*( g - A f_n), \qquad n \geq 0,
\end{equation}
to compute the minimum norm solution of the normal equation $A^* A f = A^* g$. Here, 
the operator $A^*: \Hh_2 \to \Hh_1$ denotes the adjoint of $A$ and $\omega > 0$ is a relaxation parameter.

For the initial vector $f_0 = 0$, the Landweber iterate $f_{n+1}$ belongs to the Krylov space
\[ \Kcal_n(A^* A, A^*g) := \left\{A^*g, (A^*A) A^*g, (A^*A)^2 A^*g, \cdots, (A^*A)^{n} A^*g\right\} \subset \Hh_1.\]
and can be expanded as
\begin{equation} \label{equation-Landweberpolynomial}
f_{n+1} = p_n(\omega A^*A) \omega A^* g \qquad \text{with} \qquad p_n(y) = \frac{1-(1-2y)^{n+1}}{y}.
\end{equation}
Then, for the minimum norm solution $f$ of $A^* A f = A^* g$, we get (cf. \cite[formula (6.12)]{EnglHankeNeubauer}):
\begin{equation} \label{equation-residualpolynomial}
 f - f_{n+1} = r_{n+1}(\omega A^* A) f \qquad \text{with} \qquad r_{n+1}(y) = (1 - 2 y)^{n+1}.
\end{equation}
We remark, that in contrast to other references we use an additional scaling factor $2$ in the definitions \eqref{equation-Landweberpolynomial}
and \eqref{equation-residualpolynomial} of the Landweber polynomials $p_n$ and $r_n$. This particular scaling ensures that the residual polynomials $r_{n}(y)$ converge 
pointwise to zero precisely in the interval $]0,1[$. So, if $\omega \|A^* A\| < 1$ holds, the spectral decomposition of the positive semidefinite operator 
$A^* A$ ensures the convergence of the iterate $f_n$ to $f$ (see \cite[Theorem 4.1 and Theorem 6.1]{EnglHankeNeubauer}). 

A disadvantage of the Landweber scheme is the slow convergence of $r_n(y) = (1 - 2 y)^{n}$ to zero if $y$ is close to zero. To circumvent this drawback, it is favorable to substitute the polynomials $p_n$ and $r_{n}$ of the Landweber iteration with more suitable ones. These more sophisticated iteration schemes are commonly known as semi-iterative or accelerated Landweber methods (see \cite{Egger2006}, \cite[Chapter 6]{EnglHankeNeubauer}, \cite{Hanke1991}, \cite{Hanke1996}, \cite[Section 5.2]{Rieder} and \cite{Schock1987}).
In principle, in formula \eqref{equation-Landweberpolynomial} one could use every sequence $p_n$ of polynomials with exact degree $n$
as a semi-iterative method. However, in order to evaluate the iteration polynomials $p_n$ 
and the respective residual polynomials $r_{n+1}(y) := 1 - y p_n(y)$ in a cost-effective way, it is advantageous to use sequences of orthogonal polynomials (see \cite{Hanke1991}).

If $P_n(x)$, $n \in \Nn_0$, denote monic polynomials of degree $n$ orthogonal with respect to a weight function $w$ supported on the reference interval $[-1,1]$, the polynomials $P_n$
can be evaluated cheaply by the three-term recurrence relation (cf. \cite[I. Theorem 4.1]{Chihara})
\begin{equation} \label{equation-3termrecurrencegeneral} 
P_{n+1}(x) = (x - \alpha_n) P_n(x) - \beta_n P_{n-1}(x), \quad P_0(x) = 1, \quad P_1(x) = x - \alpha_0.
\end{equation}
It is well-known that the coefficients $\alpha_n \in \Rr$ and $\beta_n > 0$ are uniquely given and that $P_n(1) > 0$ holds for all $n \in \Nn_0$. 
Now, if we define the residual polynomials $r_n$ on $[0,1]$ by
\begin{equation} \label{equation-residualnumethod} r_n(y) := \frac{P_n (1-2 y)}{P_n(1)},\end{equation}
the constraint $r_n(0) = 1$ is satisfied and \eqref{equation-3termrecurrencegeneral} yields the following recurrence formula:
\begin{align} \label{equation-Residualpolynomials} 
r_{n+1}(y) &= (1 - \alpha_n - 2 y) \frac{  P_{n}(1) }{P_{n+1}(1)} r_n(y) - \beta_n \frac{  P_{n-1}(1) }{P_{n}(1)} \frac{  P_{n}(1) }{P_{n+1}(1)} r_{n-1}(y), \quad n \geq 1,\\
r_{0}(y) &= 1, \qquad r_1(y) = 1 - \frac{2}{1-\alpha_0} y. \notag 
\end{align}
Also the coefficients $\mu_{n+1} = \frac{P_n(1)}{P_{n+1}(1)}$ can be computed recursively via \eqref{equation-3termrecurrencegeneral} as
\[\mu_{n+1} = \frac{P_n(1)}{P_{n+1}(1)} = \frac{P_n(1)}{(1-\alpha_n)P_n(1) - \beta_n P_{n-1}(1)} = \frac{1}{(1-\alpha_n)-\beta_n \mu_n}.\]
The resultant recursion formula for the iteration polynomials $p_n(x) = \frac{1 - r_{n+1}(x)}{x}$ yields the following semi-iterative algorithm (stated with a slightly different notation in \cite{Hanke1991}):

\begin{algorithm} 
\caption{Semi-iterative method based on monic orthogonal polynomials on $[-1,1]$}
\label{algorithm-2}

\begin{algorithmic}[H]
\STATE $\mu_1 = \frac{1}{1 - \alpha_0}$
\STATE $f_0 = 0$, $f_1 = 2 \mu_1 \; \omega A^* g $
\WHILE {(stopping criterion false)} 
\STATE $\mu_{n+1} = \frac{1}{1-\alpha_n - \beta_n \mu_n }$ \\[2mm]
\STATE $f_{n+1} = f_n + ((1-\alpha_n) \mu_{n+1}-1) ( f_n - f_{n-1}) + 2 \mu_{n+1} \; \omega  A^*( g - A f_n)$\\[2mm]
\STATE $n \to n+1$
\ENDWHILE
\end{algorithmic}
\end{algorithm}

Setting $\alpha_n = \beta_n = 0$, Algorithm \ref{algorithm-2} describes the Landweber iteration \eqref{equation-Landeber}. Other well-known examples 
of Algorithm \ref{algorithm-2} are based on the Chebyshev polynomials $U_n$ of the second kind and the Jacobi polynomials $P_n^{(\nu - 1/2,-1/2)}$, $\nu > 0$. In the first case, the scheme is known as Chebyshev method of Stiefel, in the second case as the $\nu$-methods of Brakhage \cite{Brakhage1987}. For $\nu = 1$, the scheme is known as Chebyshev method of Nemirovskii and Polyak (see \cite{NemirovskiiPolyak1984}). We remark that in this article the parameter $\nu$ of the $\nu$-methods is set twice as large as normally used in the literature. In this way, the parameter $\nu$ coincides with the parameter of the ultraspherical polynomials. As a stopping criterion for Algorithm \ref{algorithm-2}, several choices are possible (see \cite{Hanke1991}). However, the most common one is certainly the discrepancy principle of Morozov and some generalizations of it. 

To analyse the convergence of Algorithm \ref{algorithm-2}, we consider smooth solutions $f$ of \eqref{equation-illposedLGS} in subspaces $X_s := \range((A^*A)^{\frac{s}{2}})$, $s \geq 0$,
of the Hilbert space $\Hh_1$ and the moduli of convergence
\[\eps_s(n) = \sup_{y \in [0,1]} |y^{\frac{s}{2}} r_n(y)|, \qquad \eps_s^S(n) = \sup_{y \in [0,1]} |y^{\frac{s}{2}}(1-y)^{\frac{s}{2}} r_n(y)|, \qquad n \in \Nn. \]
We will use the modulus $\eps_s(n)$ if the residual polynomial $r_n(y)$ converges to zero at $y = 1$ and $\eps_s^S(n)$ otherwise. 
In the first case, if $f = (\omega A^*A)^{\frac{s}{2}} h \in X_s$ holds with $h \in \Hh_1$ and $\omega \| A^* A \| \leq 1$, the spectral theorem yields the error estimate (see \cite[Theorem 3.2]{Hanke1991})
\[ \|f - f_n \| \leq \eps_s(n) \|h\|.\]
In the second case, we assume that $\omega \| A^* A \| < 1$ holds such that $\id - \omega A^* A$ is an invertible operator on $\Hh_1$. Then, a similar argumentation as in \cite[Theorem 3.2]{Hanke1991} yields the bound
\[ \|f - f_n \| \leq \frac{\eps_s^S(n)}{(1-\omega \|A^*A\|)^\frac{s}{2}} \|h\|. \]
The convergence rate of the Landweber method is known to be of order $\eps_s^S(n) = \Ord(n^{-\frac{s}{2}})$, while the $\nu$-methods reveal $\eps_{s}(n) = \Ord(n^{-s})$ for $0 < s \leq \nu$ (see
\cite[Chapter 6]{EnglHankeNeubauer},\cite{Hanke1991}). 

The main goal of this article is to find and investigate orthogonal polynomials $P_n$ with a priori given recurrence coefficients $\alpha_n$ and $\beta_n$
such that the resulting semi-iterative scheme in Algorithm \ref{algorithm-2} improves the performance of the $\nu$-methods. More precisely, under some general assumptions on the given data $A$ and $g$, we want to determine new semi-iterative methods in which the error $\|A f_n - g\|$ in Algorithm \ref{algorithm-2} is smaller compared to the $\nu$-methods. If Algorithm \ref{algorithm-2} is
stopped according to the discrepancy principle, this will result in an earlier termination of the iteration. 
However, a well-known result (cf. \cite[Theorem 4.1]{Hanke1991}, \cite{Hanke1996}) states that the convergence order $\eps_{s}(n) = \Ord(n^{-s})$, $0 < s \leq \nu$, of the $\nu$-methods is already optimal and that it is not possible to obtain semi-iterative methods with a better order. 

Searching for new semi-iterative schemes, we focus therefore on a different important aspect: the decay of the residual polynomials $r_n(y)$ at $y = 0$. For linear ill-posed problems, the operator $A^*A$ on $\Hh_1$ is typically compact and its spectrum is clustered at the origin $y=0$. 
Thus, if we assume that most of the spectrum of $A^*A$ is concentrated at $y=0$, the error $\|A f_n - g\|$ in Algorithm \ref{algorithm-2} depends strongly on how fast the residual polynomials $r_n(y)$ decay to zero in the neighborhood of $y = 0$. For this reason, various concepts of fast decaying polynomials have already been studied, see \cite{Hanke1996} and the references therein.

If the residual polynomials $r_n$ are orthogonal, the decay of $r_n(y)$ at $y=0$ is directly linked to the location of the smallest root of $r_n(y)$ in the interval $[0,1]$. The closer to the origin the smallest root is, the faster $r_n(y)$ decays at $y=0$. This link is now used to construct polynomials $r_n$ with a faster decay at $y=0$ than the residual polynomials of the $\nu$-methods. To this end, we alter particular coefficients in the recurrence relation of the orthogonal polynomials linked to the $\nu$-methods. This altering leads directly to a family of co-dilated orthogonal polynomials, many of whose characteristics are known in the literature, see \cite{DiniMaroniRonveaux1989, IfantisSiafarikas1995-2, MarcellanDehesaRonveaux1990, RonveauxBelmehdiDiniMaroni1990, Slim1988}. Based on these co-dilated orthogonal polynomials, we will construct the new semi-iterative methods and investigate some of their properties.

The main idea of this article is explained in more detail in the next section on the basis of the Chebyshev polynomials.
In Section $3$, the theoretical fundamentals of the co-dilated orthogonal polynomials are laid and some properties of their extremal roots are investigated. 
Section $4$ is devoted to the particular family of co-dilated ultraspherical polynomials and their properties. In Section $5$, the transition from the co-dilated ultraspherical polynomials to the co-dilated $\nu$-methods is illustrated. The main results of Section $4$ and $5$, formulated in Theorem \ref{Theorem-convergenceultraspherical} and Corollary \ref{Corollary-convergencenonsymmetric}, state that under certain conditions on the dilation parameter the convergence order of the new semi-iterative methods is the same as for the $\nu$-methods. In Section $6$, it is shown how for the co-dilated $1$-method the dilation parameter can be fitted optimally to minimize the error $\|A f_n - g\|$. Finally, in the last section some numerical tests are conducted. 

In this article, the coefficients $\alpha_n$ and $\beta_n$ in Algorithm \ref{algorithm-2} are always a priori given. Disabling this constraint, a powerful alternative is given by the method of conjugate gradients where $\alpha_n$ and $\beta_n$ depend on $A$ and $g$ (see \cite[Chapter 7]{EnglHankeNeubauer}, \cite{Fischer}, \cite[Section 5.3]{Rieder}). It is well-known that the iterate $f_{n}$ of the cg-algorithm minimizes the error $\| A f_{n} - g \|$ in the Krylov space $\Kcal_{n-1}( A^*A,  A^*g)$. On the other hand, the cg-iteration has a multifarious convergence behavior that makes it harder to handle as a regularization tool than the $\nu$-methods. A deep analysis of the cg-algorithm as a regularization tool and a comparison with the $\nu$-methods can be found in \cite[Chapter 7]{EnglHankeNeubauer}, \cite{Hanke1991} and \cite[Section 5.3]{Rieder}.

\section{Co-dilated Chebyshev polynomials}

We illustrate the idea of this article on the basis of the Chebyshev polynomials of the second kind.
For $x = \cos t$, $t \in [0,\pi]$, the monic Chebyshev polynomials of the first and the second kind are explicitly given as (cf. \cite[p. 28]{Gautschi})
\begin{align*}
T_n(\cos t) &= \frac{1}{2^{n-1}} \cos n t, \qquad U_n(\cos t) = \frac{1}{2^n}\frac{\sin (n+1)t}{\sin t}. 
\end{align*}
Further, we consider linear combinations of $U_n$ and $T_n$, i.e.
\begin{align} \label{equation-definitioncodilatedchebyshev} 
U_n^*(x) \equiv U_n^*(x,\lambda) &:= (2 - \lambda) U_n(x) + (\lambda-1) T_n(x) \\
                                 &\phantom{:}= \lambda U_n(x) + (1-\lambda) x U_{n-1}(x), \notag \\
                                 &\phantom{:}= U_n(x) + \frac{1-\lambda}{4} U_{n-2}(x), \qquad \lambda \in \Rr. \notag
\end{align}
The last two identities in \eqref{equation-definitioncodilatedchebyshev} follow from simple trigonometric conversions.  
In order to use these polynomials in a semi-iterative scheme, we introduce according to \eqref{equation-residualnumethod} the residual polynomials
\[r_n^*(y) \equiv r_n^*(y,\lambda) = \frac{U_n^*(1-2y)}{U_n^*(1)}, \qquad y \in [0,1]. \]
In Figure \ref{figure-Chebyshevpolynomials}, the normalized polynomials $U_6(x)/U_6(1)$, $T_6(x)/T_6(1)$ and $U_6^*(x)/U_6^*(1)$ with $\lambda = 1.5$ are plotted. 
For $-1 < x < 1$, the polynomials $U_n(x)/U_n(1)$ converge pointwise to zero as $n \to \infty$. Thus, also the residual polynomials $r_n^*(y,1)$ converge pointwise to zero for $y \in ]0,1[$. 
They form a convergent semi-iterative scheme, the so-called Chebyshev method of Stiefel (cf. \cite[p. 116]{Rieder}). 
On the other hand, the polynomials $T_n(x)/T_n(1)$ do not converge pointwise to zero on $[-1,1]$. Nevertheless, the largest root of $T_n$ is much closer to $x = 1$ than the corresponding root of the polynomial $U_n$. This implies that the smallest root of the respective residual polynomial $r_n^*(y,2)$ is closer to $y = 0$ and that $r_n^*(y,2)$ decays faster at $y = 0$ than the polynomial $r_n^*(y,1)$.

\begin{figure} 
\begin{center} 
        \begin{minipage}[t]{0.32\linewidth} 
        \begin{center}
        \includegraphics[angle=-90,width=\linewidth]{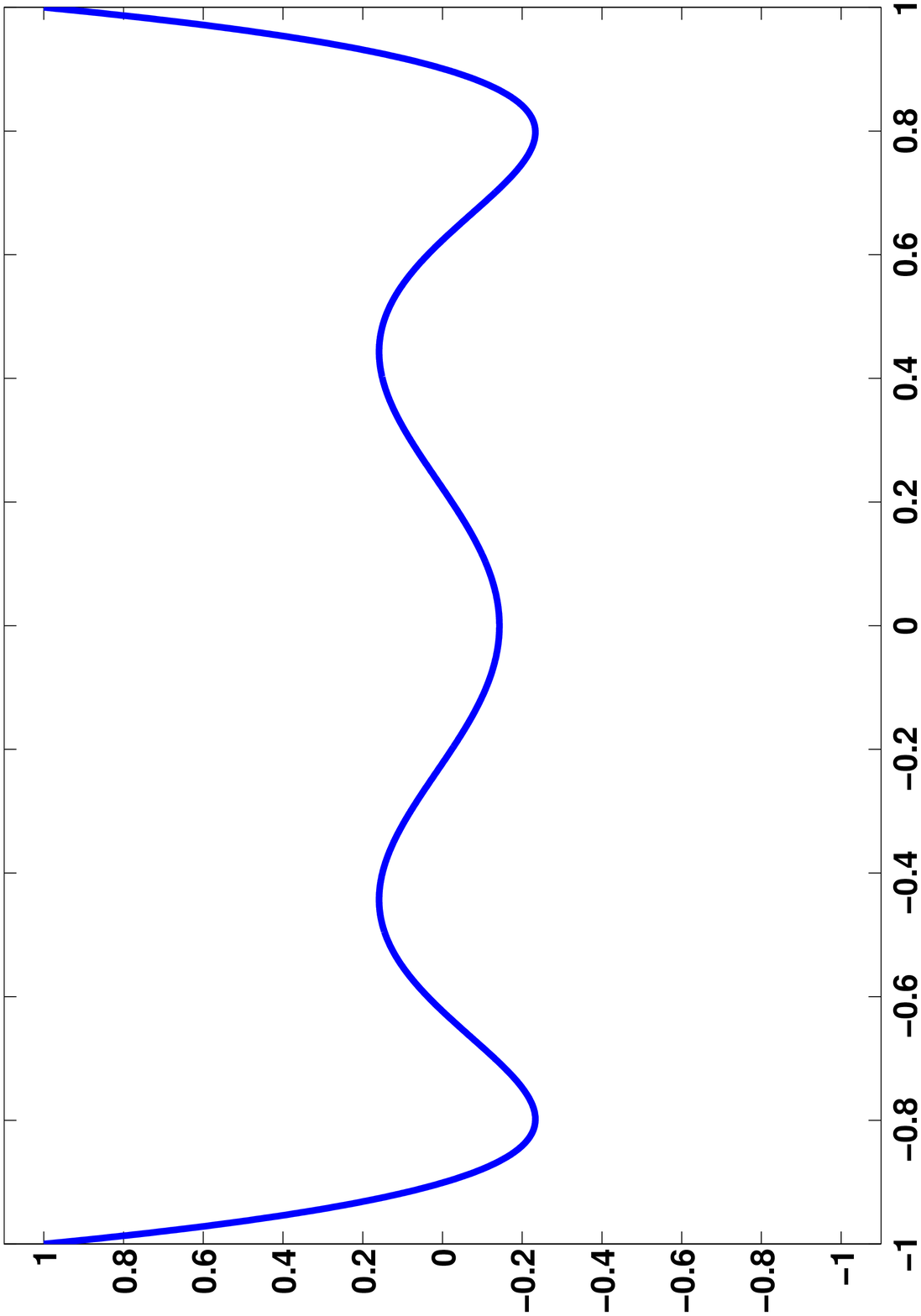}\\ $U_6(x)/U_6(1)$ 
        \end{center}
        \end{minipage}
	\begin{minipage}[t]{0.32\linewidth}
        \begin{center}
        \includegraphics[angle=-90,width=\linewidth]{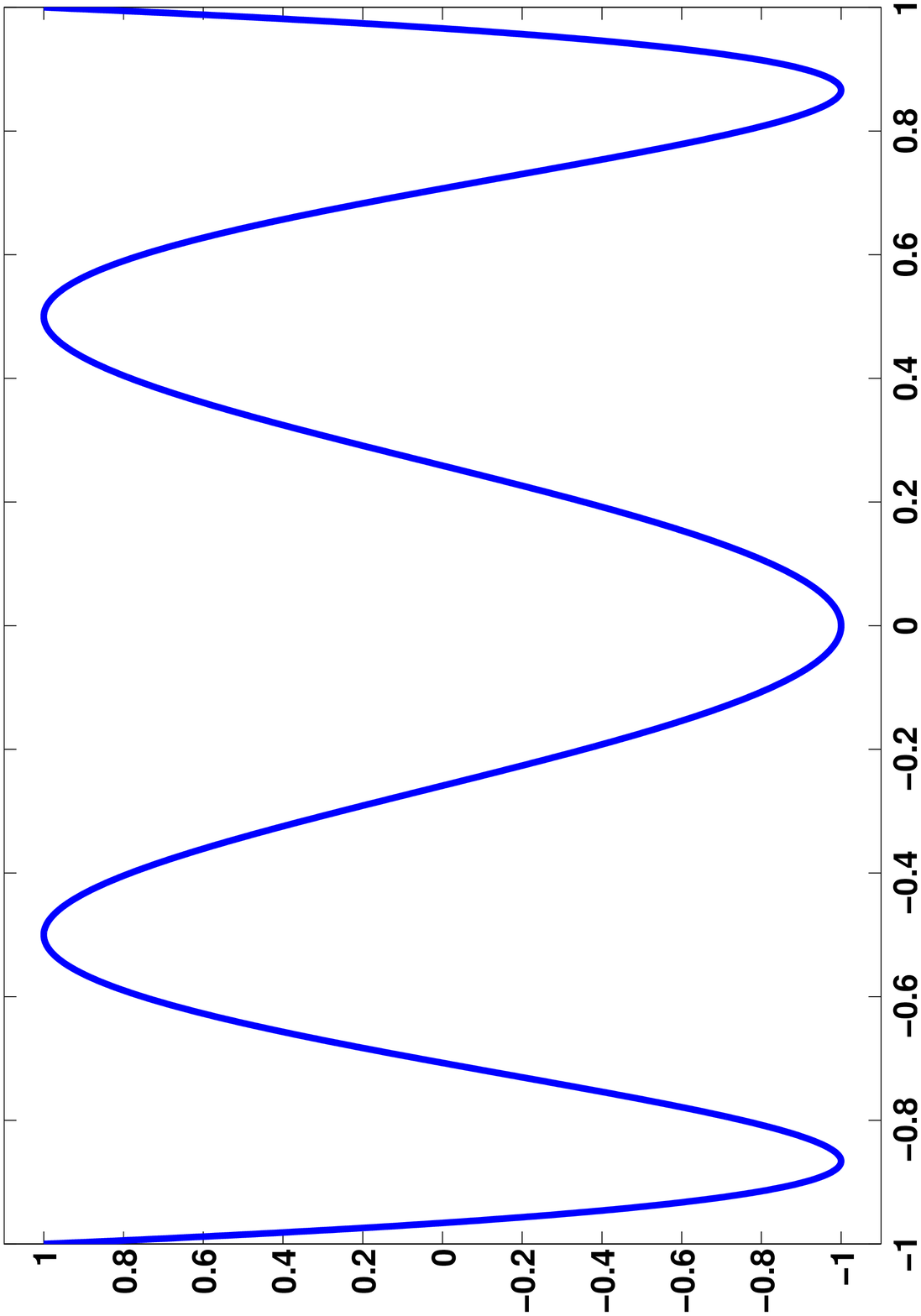}\\ $T_6(x)/T_6(1)$ 
        \end{center}
	\end{minipage}
	\begin{minipage}[t]{0.32\linewidth}
        \begin{center}
        \includegraphics[angle=-90,width=\linewidth]{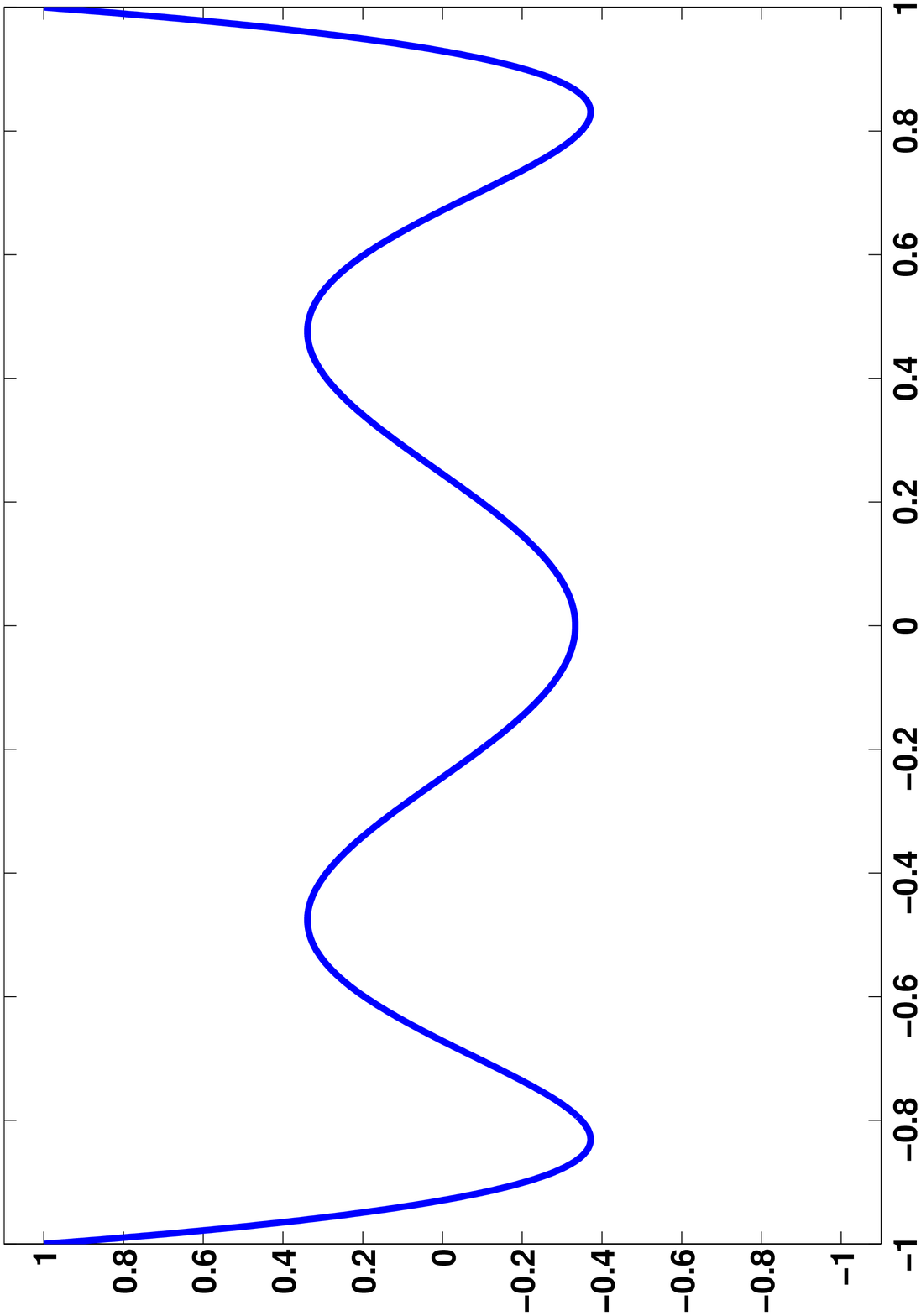}\\ $U_6^*(x)/U_6^*(1)$, $\lambda = 1.5$. 
        \end{center}
	\end{minipage}
\caption{The Chebyshev polynomials $U_6/U_6(1)$, $T_6/T_6(1)$ and the linear combination $U_6^*/U_6^*(1)$, $\lambda = 1.5$, on the interval $[-1,1]$. }
\label{figure-Chebyshevpolynomials}
\end{center}
\end{figure}

Thus, although not giving a convergent iterative scheme, the residual polynomial $r_n^*(y,2)$ has the favorable property to decay fast at the origin. In order to combine both requests, a convergent scheme
and a fast decay at $y=0$, we consider now the linear combinations $U_n^*$ of $U_n$ and $T_n$. With the identities (see \cite[Section A.2, A.3]{Boyd})
\begin{equation} 
\label{equation-identitiesat1} U_n(1) = \frac{n+1}{2^n}, \quad T_n(1) = \frac{1}{2^{n-1}}, \quad U_n'(1) = \frac{ n(n+1)(n+2) }{ 2^{n} 3}, \quad  T_n'(1) = \frac{n^2}{2^{n-1}},
\end{equation}
we get by a simple computation the following formula for the derivative of $U_n^*$ at $x = 1$:
\begin{align*}
\frac{{U_n^{*}}'(1)}{U_n^*(1)} &= \frac{(2-\lambda) U_n'(1) + (\lambda-1) T_n'(1)}{(2-\lambda) U_n(1) + (\lambda-1) T_n(1)} \\
&= \frac{(2-\lambda) \frac{n(n+1)(n+2)}{3} + 2(\lambda-1) n^2}{(2-\lambda) (n+1) + 2(\lambda-1)} = \frac{n^2+2}{3} + \frac{2}{3}\frac{\lambda(n^2-1)}{(2-\lambda)n+\lambda}.
\end{align*}
Hence, for $\lambda <2 $, ${U_n^{*}}'(1)/U_n^*(1)$ is an increasing function of the parameter $\lambda$. Therefore, also the decay ${r_n^*}'(0) = - {U_n^{*}}'(1)/U_n^*(1)$ of the residual polynomials at $y = 0$ gets faster with increasing $\lambda$. So, we can conclude that for $1< \lambda \leq 2$ the residual polynomials $r_n^*(y,\lambda)$ have a faster decay at $y = 0$ than the residual polynomials $r_n^*(y,1)$ 
of the Chebyshev method. On the other hand, it is visible in Figure \ref{figure-Chebyshevpolynomials} that the oscillations of the polynomial $U_n^*(x)$, $\lambda = 1.5$, in the interval $[-1,1]$ have a larger amplitude compared to the polynomial $U_n(x)$. This holds generally for $1<\lambda<2$ and is also visible in the following convergence result. 

\begin{Thm} \label{Theorem-convergenceChebyshev} 
For $\lambda < 2$, $x \in [-1,1]$, the polynomial $U_n^*(x)$, $n \geq 3$, is bounded by
\begin{equation} \label{equation-chebysheverror1} \left| \frac{U_n^*(x)}{U_n^*(1)} \right| \leq \frac{1+|1-\lambda|}{2-\lambda} \frac{ 2 }{\frac12 + \sqrt{1-x^2} (n-1)}. \end{equation}
For the residual polynomial $r_n^*(y) = \frac{U_n^*(1-2y)}{U_n^*(1)}$ on $[0,1]$ and the modulus of convergence $\eps_1^S(n)$, we get the estimates
\begin{equation} \label{equation-chebysheverror2} |r_n^*(y)| \leq \frac{1+|1-\lambda|}{2-\lambda} \frac{1}{\frac14 + \sqrt{y(1-y)} (n-1)}, \qquad \eps_1^S(n) \leq \frac{1+|1-\lambda|}{2-\lambda} \frac{1}{n-1}. \end{equation}
\end{Thm}

\begin{proof}
By the third identity in \eqref{equation-definitioncodilatedchebyshev}, we get for $\lambda < 2$ and $x = \cos t$, $t \in [0,\pi]$, the bound
\begin{align*}
\left( \sqrt{1-x^2}+\frac{1}{n+1}\right) \left| \frac{U_n^*(x)}{U_n^*(1)}\right| & \leq
\left( \sin t+\frac{1}{n+1}\right) \left( \frac{ \left| \frac{\sin ((n+1) t)}{\sin t}\right| +  |1-\lambda| \left| \frac{\sin ((n-1) t)}{\sin t}\right|}{(2-\lambda)n + \lambda} \right)\\
& \leq \frac{2 (1+|1-\lambda|) }{(2-\lambda)(n-1) + 2} \leq \frac{2 (1+|1-\lambda|) }{(2-\lambda)(n-1)}.
\end{align*}
Dividing both sides by $\left(\sqrt{1-x^2}+\frac{1}{n+1}\right)$, we can conclude for $n \geq 3$:
\begin{align*}
 \left| \frac{U_n^*(x)}{U_n^*(1)}\right| &\leq \frac{2 (1+|1-\lambda|) }{(2-\lambda)(n-1)  \left( \sqrt{1-x^2}+\frac{1}{n+1}\right)} 
\leq \frac{1+|1-\lambda|}{2-\lambda}\frac{ 2 }{ \frac12+ (n-1) \sqrt{1-x^2}}.
\end{align*}
The estimates for the residual polynomials $r_n^*(y)$ and the modulus $\eps_1^S(n)$ follow immediately from the estimate of $U_n^*(x)$.
\end{proof}

Theorem \ref{Theorem-convergenceChebyshev} states that for all $\lambda < 2$ the symmetric modulus $\eps_1^S(n)$ has the same order $\Ord(n^{-1})$ of convergence.
However, the factor $\frac{1+|1-\lambda|}{2-\lambda}$ in \eqref{equation-chebysheverror2} has a considerable impact on the error estimates if $\lambda$ is
close to $2$. If $\lambda = 2$, the polynomials $U_n^*$ correspond to the Chebyshev polynomials $T_n$ of the first kind and the corresponding semi-iterative scheme is not convergent. 

We have seen so far that for $1 < \lambda < 2$ the residual polynomials $r_n^*(y,\lambda)$ decay faster at $y=0$ but implicate slightly larger error bounds in 
Theorem \ref{Theorem-convergenceChebyshev} than the residual polynomials $r_n^*(y,1)$ linked to the Chebyshev polynomials $U_n$.
It depends now on the given operator $A$ and the right hand side $g$, whether it is favorable
to choose $U_n$ or $U_n^*$ in a semi-iterative scheme. Assuming that most of the spectrum of $A^*A$ is concentrated at $y = 0$, the choice of $U_n^*$ in Algorithm \ref{algorithm-2} with an appropriate $\lambda > 1$ can have advantages compared to $U_n$. 

Finally, we derive the recurrence coefficients of the semi-iterative scheme based on the polynomials $U_n^*$. For the polynomials $U_n^*$, we have first the three-term recurrence relation
\begin{align} \label{equation-3termrecurrencechebyshev}
U_{n+1}^*(x) &= x U_n^*(x) - \frac{1}{4} U_{n-1}^*(x), \qquad n \geq 2 \\ 
U_0^*(x) &= 1, \qquad U_1^*(x) = x, \qquad U_2^*(x) = x^2 - \lambda \frac{1}{4}. \notag
\end{align}
It is well-known that the monic polynomials $T_n$ and $U_n$ satisfy \eqref{equation-3termrecurrencechebyshev} with $\lambda = 2$ and $\lambda = 1$. Then, it follows immediately that \eqref{equation-3termrecurrencechebyshev} holds
for the linear combination $U_n^*$. 

In view of \eqref{equation-3termrecurrencechebyshev}, the polynomials $U_n^*$ turn out to be a particular family of co-dilated orthogonal polynomials constructed by dilating 
a coefficient in the recurrence relation of the polynomials $U_n$ by a factor $\lambda$. This special construction and the consequences regarding the roots of $U_n^*$ are investigated in more detail in the next section. In the following, the polynomials $U_n^*$ are referred to as co-dilated Chebyshev polynomials. 

The coefficients $\mu_{n+1}$ in Algorithm \ref{algorithm-2} can also be computed explicitly as
\[ \mu_{n+1} = \frac{U_n^*(1)}{U_{n+1}^*(1)} = 2 \frac{(2-\lambda)(n+1) + 2\lambda-2}{(2-\lambda)(n+2) + 2\lambda-2} = 2 \frac{(2-\lambda)n + \lambda}{(2-\lambda)n + 2}.\]
Therefore, using the recurrence coefficients of the polynomials $U_n^*$ in Algorithm \ref{algorithm-2}, we get the following recurrence formula for the iterates:
\begin{align*}
f_{n+1} &= f_n + \frac{(2-\lambda)n+2\lambda-2}{(2-\lambda)n+2}\;(f_n - f_{n-1}) + 4 \,\frac{(2-\lambda)n+\lambda}{(2-\lambda) n+2}\, \omega \;A^*( g - A f_n), \quad n \geq 1, \\
      f_1 &= 2 \omega \; A^* g, \qquad f_0 = 0.
\end{align*}
For $\lambda = 1$, this iteration corresponds precisely with the Chebyshev method of Stiefel (see \cite[p. 116]{Rieder}). 

\section{Symmetric co-dilated orthogonal polynomials}

In this section, we generalize the concept of the co-dilated Chebyshev polynomials to arbitrary symmetric orthogonal polynomials on the interval $[-1,1]$. 
We denote by $P_n$ the monic polynomials of degree $n$ orthogonal with respect to an axisymmetric 
weight function $w$ supported on $[-1,1]$. In this case, the coefficients $\alpha_n$, $n \in \Nn_0$, in \eqref{equation-3termrecurrencegeneral} vanish and we obtain the three-term recurrence relation
\begin{equation} \label{equation-3termrecurrence} 
P_{n+1}(x) = x P_n(x) - \beta_n P_{n-1}(x), \quad P_0(x) = 1, \quad P_1(x) = x,
\end{equation}
with positive coefficients $\beta_n > 0$, $n \in \Nn$. The monic co-dilated orthogonal polynomials $P_n^*(x) \equiv P_n^*(x,\lambda,m)$
are now derived from the original polynomials $P_n$ on $[-1,1]$ by dilating the coefficient $\beta_m$ in the three-term recurrence relation by a factor $\lambda \in \Rr$.
\begin{align} \label{equation-3termrecurrencedilated}
P_0^*(x) &= 1, \qquad P_1^*(x) = x, \notag \\
P_{n+1}^*(x) &= x P_n^*(x) - \beta_n P_{n-1}^*(x), \quad n \neq m, \\
P_{m+1}^*(x) &= x P_m^*(x) - \lambda \beta_m P_{m-1}^*(x). \notag
\end{align}
If $\lambda > 0$, Favards Theorem ensures that $P_n^*(x)$, $n \in \Nn$, is a family of 
orthogonal polynomials. For $m = 1$, the co-dilated orthogonal polynomials $P_n^*(x)$ 
were firstly introduced in \cite{Dini} by Dini and then generalized in \cite{DiniMaroniRonveaux1989, RonveauxBelmehdiDiniMaroni1990}. Many properties of the zeros of the co-dilated 
orthogonal polynomials like interlacing behavior and the
distribution of the zeros are well-known and studied in \cite{IfantisSiafarikas1995-2}, \cite{MarcellanDehesaRonveaux1990} and \cite{Slim1988}. 
We will add some more properties in the course of this section. 

First of all, the co-dilated polynomials can be represented with help of the numerator polynomials associated to $P_n$ (see \cite{MarcellanDehesaRonveaux1990}). Therefore, we denote by $P_n^{(m)}(x)$ the $m$-th. numerator polynomials of $P_n$ defined by the shifted recursion formula
\begin{align} \label{equation-3termrecurrencenumerator} 
P_{n+1}^{(m)}(x) &= x P_n^{(m)}(x) - b_{n+m} P_{n-1}^{(m)}(x), \quad n \geq 1, \\ P_0^{(m)}(x) &= 1, \quad P_1^{(m)}(x) = x. \notag
\end{align}
Then, by a simple induction argument, the co-dilated polynomials $P_n^*$ can be written as 
\begin{align} \label{equation-characteristiccodilated}
P_n^*(x) & =  \lambda P_n(x) + (1-\lambda) P_m(x) P_{n-m}^{(m)}(x),\qquad \text{for $n>m$,} \\
P_n^*(x) & = P_n(x), \hspace{5.2cm} \text{for $n \leq m$.} \notag
\end{align}


Now, we investigate the behavior of the zeros of $P_n^*$ if the dilation parameter $\lambda$ in the recurrence relation \eqref{equation-3termrecurrencedilated} is altered. 
We denote by $x_{n,j}$ and $x_{n,j}^*$, $1 \leq j \leq n$, the $n$ zeros of $P_n$ and $P_n^*$ in ascending order. 
Then, we get the following result for the extremal roots of $P_n^*$ and $P_n$. 

\begin{Thm} \label{Theorem-largerzeros}
The largest zero $x_{n,n}^*$ of $P_n^*(x)$ is a monotone increasing function of the dilation parameter $\lambda$, 
the smallest zero $x_{n,1}^*$ is a monotone decreasing function of $\lambda$. 
In particular, for $\lambda> 1$, we have
\begin{align*}
x_{n,n}^* = x_{n,n} \qquad \text{and} \qquad x_{n,1}^* = x_{n,1} &\qquad \text{for \quad $n \leq m$}, \\
x_{n,n}^* > x_{n,n} \qquad \text{and} \qquad x_{n,1}^* < x_{n,1} &\qquad \text{for \quad $n > m$}.
\end{align*}
\end{Thm}

\begin{proof}
We deduce Theorem \ref{Theorem-largerzeros} from a general result on the monotonicity of the extremal zeros based on the Hellman-Feynman theorem (see \cite{ErbTookos2011}, \cite[Section 7.3 and 7.4]{Ismail}
and the references therein). This general result states that if the coefficients $\alpha_n$ and $\beta_n$ are differentiable monotone increasing functions of the parameter $\lambda$, then also the largest root is an increasing function of $\lambda$. In our case, the derivatives of the coefficients $\alpha_n$ and $\beta_n$ with respect to the dilation parameter $\lambda$ are given by
\begin{align*} \alpha_n'(\lambda) &= 0, \quad n \in \Nn_0, \\
               \beta_n'(\lambda) & = 0,  \quad n \neq m, \qquad \beta_m'(\lambda) = \beta_m > 0
\end{align*}
and therefore nonnegative. Thus, by \cite[Theorem 1.1]{ErbTookos2011} the largest root $x_{n,n}^*$ of $P_n^*(x)$ is a monotonic increasing function of $\lambda$. For $n \leq m$, equation \eqref{equation-characteristiccodilated} implies that the polynomials $P_n$ and $P_n^*$ and therefore also the zeros $x_{n,n}$ and $x_{n,n}^*$ coincide. For 
$n > m$, the Hellman-Feynman theorem implies the formula (see \cite[formula (2)]{ErbTookos2011}, \cite[formula (7.3.8)]{Ismail})
\[ \frac{ \mathrm{d} x_{n,n}^*}{\mathrm{d} \lambda} = l_{n,n}^*
\frac{\beta_m P_m^*(x_{n,n}^*) P_{m-1}^*(x_{n,n}^*)}{\lambda \beta_1 \beta_2 \cdots \beta_{m}} > 0\]
for the derivative of $x_{n,n}^*$ with respect to the parameter $\lambda$, where $l_{n,n}^*>0$ denotes the Christoffel number corresponding to the zero $x_{n,n}^*$. So, for $n > m$ and $\lambda > 1$, the root $x_{n,n}^*$ is strictly larger than $x_{n,n}$. By the symmetry of the polynomials $P_n^{*}$, we have $x_{n,1}^* = - x_{n,n}^*$. This 
implies the statement for the smallest roots.  
\end{proof}

\begin{Rem}
Alternatively, it is also possible to prove Theorem \ref{Theorem-largerzeros} with the Perron-Frobenius theory. One way to do this consists in adopting \cite[Theorem 7.4.1]{Ismail}
to the setting of Theorem \ref{Theorem-largerzeros}. For the case $m = 1$, even stronger statements can be shown. For $\lambda > 1$, Slim proved in \cite{Slim1988} the following interlacing properties for the zeros $x_{n,j}^*$ and $x_{n,j}$:
\begin{align*}
x_{n,j}^* < x_{n,j} < x_{n,j+1}^*, &\qquad \text{for} \quad 1 \leq j \leq \left\lfloor \frac{n}{2} \right\rfloor, \\
x_{n,j-1}^* < x_{n,j} < x_{n,j}^*, &\qquad \text{for} \quad \left\lceil \frac{n}{2} \right\rceil < j \leq n.
\end{align*}
\end{Rem}

In order to get residual polynomials that are small in the interior of $[0,1]$,
it is important that all the zeros of the co-dilated polynomials $P_n^*$ are in the interior of the interval $[-1,1]$.
Restricting the dilation parameter $\lambda$ appropriately, this can indeed be proven. 

\begin{Lem} \label{Lemma-zeroscodilated}
All the zeros of the polynomials $P_n^*(x)$, $n \in \Nn$, are in the interior of $[-1,1]$, if and only if
\begin{equation} \label{equation-zeroscodilated}
\lambda \leq \frac{1}{1 - L_m},\end{equation}
with the constant $L_m$ given by 
\[ 0 \leq L_m :=  \lim_{n \to \infty} \frac{P_n(1)}{P_{n-m}^{(m)}(1)} \frac{1}{P_m(1)}  < 1. \]
\end{Lem}

\begin{proof}
By induction we prove the following identity for the numerator polynomials $P_{n}^{(m)}$: 
\begin{equation} \label{equation-identitynumerator}
P_{n+1}(x) P_{n-m}^{(m)}(x) - P_{n-m+1}^{(m)}(x) P_n(x) = - \beta_m \beta_{m+1} \cdots \beta_{n} P_{m-1}(x), \qquad n \geq m.
\end{equation}
For $n = m$, this is clearly the recurrence formula for $P_{m+1}$. Assuming that \eqref{equation-identitynumerator} holds for an integer $n > m$, we show that 
\eqref{equation-identitynumerator} holds also for $n+1$. To this end we adopt the three-term recurrence formulas \eqref{equation-3termrecurrence} and \eqref{equation-3termrecurrencenumerator} and get:
\begin{align*}
P_{n+2}(x) P_{n-m+1}^{(m)}(x) & - P_{n-m+2}^{(m)}(x) P_{n+1}(x) \\ & = x (P_{n+1}(x) P_{n-m+1}^{(m)}(x) - P_{n-m+1}^{(m)}(x) P_{n+1}(x)) \\ & \qquad - 
\beta_{n+1} (P_{n}(x) P_{n-m+1}^{(m)}(x) - P_{n-m}^{(m)}(x) P_{n+1}(x)) \\
& = - \beta_m \beta_{m+1} \cdots \beta_{n} \beta_{n+1} P_{m-1}(x).
\end{align*}
For $x \geq 1$, we have $P_m(x) > 0$ and $P_{n-m}^{(m)}(x) > 0$ for all $n \geq m \geq 0$. 
Then, using \eqref{equation-identitynumerator}, we get the following chain of inequalities for $x \geq 1$:
\begin{align} \label{equation-chaininequality}
 0 < \frac{P_{n+1}(x)}{P_{n-m+1}^{(m)}(x)} & = \frac{P_{n}(x)}{P_{n-m}^{(m)}(x)} - \frac{\beta_m \beta_{m+1} \cdots \beta_{n} P_{m-1}(x)}{P_{n-m+1}^{(m)}(x) P_{n-m}^{(m)}(x)} \notag \\&  < 
\frac{P_{n}(x)}{P_{n-m}^{(m)}(x)} < \cdots < \frac{P_{m}(x)}{P_{0}^{(m)}(x)} = P_m(x). 
\end{align}
This implies first of all that $0 \leq L_m < 1$ exists. Further, by the identity \eqref{equation-characteristiccodilated} we get
\[ \frac{P_{n}^*(1)}{P_{n-m}^{(m)}(1) P_m(1)} = \lambda \frac{P_n(1)}{P_{n-m}^{(m)}(1) P_m(1)} + 1 - \lambda \quad \left\{ \begin{array}{ll} \geq 1 & \text{if $\lambda \leq 0$}, \\ 
> \lambda (L_m -1) + 1 & \text{if $\lambda > 0$.} \end{array} \right.\]
Therefore, if $\lambda \leq \frac{1}{1-L_m}$, then $P_{n}^*(1) > 0$ and $P_n^*(x)$ does not change sign for $x \geq 1$ for all $n \geq 1$. On the other hand, if $\lambda > \frac{1}{1-L_m}$, there
exists an $n \in \Nn$ such that $P_n^*(1) < 0$. Since the polynomials $P_n^*(1)$ are monic, this implies that there exists a root of $P_n^*(x)$ larger than $1$.  
Since $P_n^*(-x) = (-1)^n P_n^*(x)$, the respective statements hold also for $x \leq -1$. 
\end{proof}

\begin{Rem}
Chihara proved in \cite{Chihara1957} similar results for families of co-recursive orthogonal polynomials. 
The statement and the proof of Lemma \ref{Lemma-zeroscodilated} are adaptions of \cite[Theorem 2]{Chihara1957} to the case of co-dilated polynomials. The case 
$m = 1$ of Lemma \ref{Lemma-zeroscodilated} is proven by Slim in \cite{Slim1988}. For $m = 1$ the formula \eqref{equation-identitynumerator} is also well-known, see \cite[Equation 4.4]{Chihara}.
\end{Rem}

The next Lemma shows that the critical point $\lambda = \frac{1}{1-L_m}$ in Lemma \ref{Lemma-zeroscodilated} is also a critical point for the asymptotic behavior of the normalizing factor $P_n^*(1)$. 

\begin{Lem} \label{Lemma-asymptoticone}
Let $L_m > 0$ and $\lambda < \frac{1}{1-L_m}$. The sequence $P_n^*(1) /P_n(1)$, $n \in \Nn$, is monotonically decreasing and its limit is given by
\begin{equation} \label{equation-asymptoticone}
\lim_{n \to \infty} \frac{P_n^*(1)}{P_n(1)} = \lambda + \frac{1-\lambda}{L_m} > 0.
\end{equation}
If $\lambda = \frac{1}{1-L_m}$, then $\lim_{n \to \infty} \frac{P_n^*(1)}{P_n(1)} = 0$. 
\end{Lem}

\begin{proof}
Using formula \eqref{equation-characteristiccodilated}, we obtain
\[ \frac{P_n^*(1)}{P_n(1)} = \frac{\lambda P_n(1)}{P_n(1)} + \frac{(1-\lambda)P_{n-m}^{(m)}(1) P_m(1)}{P_n(1)} = \lambda + (1-\lambda) \frac{P_{n-m}^{(m)}(1) P_m(1)}{P_n(1)}. \]
By \eqref{equation-chaininequality}, the sequence  $\frac{P_{n-m}^{(m)}(1) P_m(1)}{P_n(1)}$ is monotonically increasing and converges to $\frac{1}{L_m}$. Therefore, $\frac{P_n^*(1)}{P_n(1)}$ is a monotonically
decreasing sequence and for the limit $n \to \infty$ we get
\[ \lim_{n \to \infty} \frac{P_n^*(1)}{P_n(1)} = \lambda + \lim_{n \to \infty} \frac{(1-\lambda)P_{n-m}^{(m)}(1) P_m(1)}{P_n(1)} 
= \lambda +\frac{1-\lambda}{L_m} =: f(\lambda). \]
Since $0 < L_m < 1$, the function $f$ is a strictly monotone decreasing in the variable $\lambda$. Moreover, we have $f(\frac{1}{1-L_m}) = 0$. This implies
the statement of Lemma \ref{Lemma-asymptoticone}.
\end{proof}

\section{Semi-iterative methods based on co-dilated ultraspherical polynomials} \label{section-codilatedultraspherical}

As a main example of accelerated Landweber methods based on co-dilated orthogonal polynomials, we consider the ultraspherical polynomials $P_n^{(\nu)}$, $\nu > -\frac{1}{2}$, and its co-dilated relatives $P_n^{(\nu)*}$. The orthogonality weight function of the ultraspherical polynomials on $[-1,1]$ is given by the function $w_\nu(x) = (1-x^2)^{\nu-\frac{1}{2}}$ with the mass
\[ \beta_0 = \int_{-1}^1 w_\nu(x) dx = \int_{-1}^1 (1-x^2)^{\nu-\frac{1}{2}} dx = \frac{2^{2\nu}\Gamma(\nu+\frac{1}{2})^2}{\Gamma(2\nu+1)}.\]
The coefficients $\beta_n$ of the three-term recurrence relation \eqref{equation-3termrecurrence} can be written explicitly as (see \cite[p. 29]{Gautschi})
\begin{equation} \label{equation-recursioncoefficientsultraspherical} \beta_n = \frac{n (n+2\nu-1)}{4 (n+\nu)(n+\nu-1)}, \qquad n \in \Nn.\end{equation}
In this section, we will only consider co-dilated polynomials in which the first coefficient $\beta_1$ is altered, i.e. in which $m = 1$ holds. To simplify the notation we will 
use the symbol $Q_n^{(\nu)}$ to denote the first order numerator polynomials of $P_n^{(\nu)}$. They can be represented as (see \cite[Chapter III, formula (4.6)]{Chihara})
\begin{equation} \label{equation-seconddefinitionnumeratorpolynomials}
Q_n^{(\nu)}(x) :=  P_n^{(\nu)(1)}(x) = \frac{1}{\beta_0}\int_{-1}^1 \frac{P_{n+1}^{(\nu)}(x) - P_{n+1}^{(\nu)}(\xi)}{x-\xi} w_\nu(\xi) d\xi, \quad n \in \Nn_0.
\end{equation}
We are using this representation to compute the critical value $L_1$. 

\begin{Lem} \label{Lemma-criticalconstantultraspherical}
For $\nu > \frac{1}{2}$, the quotient $Q_{n-1}^{(\nu)}(1)/P_{n}^{(\nu)}(1)$ is given by
\begin{equation} \label{equation-quotientnumerator}
\frac{Q_{n-1}^{(\nu)}(1)}{P_{n}^{(\nu)}(1)} = \frac{2 \nu}{ 2 \nu -1} \left(1 - \frac{\Gamma(2 \nu) \Gamma(n+1)}{\Gamma(n+2\nu)} \right).
\end{equation}
For $-\frac{1}{2} < \nu \leq \frac{1}{2}$, the sequence $Q_{n-1}^{(\nu)}(1)/P_{n}^{(\nu)}(1)$, $n \in \Nn$, diverges.
Therefore, for the ultraspherical polynomials $P_n^{(\nu)}(x)$, the constant $L_1$ in Lemma \ref{Lemma-zeroscodilated} is given by
\[ L_1 = \left\{ \begin{array}{cl} \frac{2\nu-1}{2\nu} & \qquad \text{if \quad $\frac{1}{2} < \nu$,} \\ 0 & \qquad \text{if\quad $- \frac{1}{2} < \nu \leq \frac{1}{2}$,} \end{array}\right. \]
and the statement of Lemma \ref{Lemma-zeroscodilated} holds for all $\lambda \leq \max \{1,\, 2\nu\}$.
\end{Lem}

\begin{proof}
For $\nu > - \frac{1}{2}$, let $0 < \eps < \nu + \frac{1}{2}$. Then, the function $(1+x)^{\nu-\frac{1}{2}}(1-x)^{\nu-\frac{1}{2}-\eps}$ is integrable on $[-1,1]$ and we have (cf. \cite[equation (4.02)]{Ismail})
\begin{equation} \label{equation-integralformula1}
 \int_{-1}^1  (1+x)^{\nu-\frac{1}{2}}(1-x)^{\nu-\frac{1}{2}-\eps} dx = \frac{2^{2\nu-\eps}\Gamma(\nu+\frac{1}{2}) \Gamma(\nu-\eps + \frac{1}{2})}{\Gamma(2\nu-\eps+1)}.
\end{equation}
More generally, using the Rodriguez formula \cite[(4.7.12)]{Szegö} for the ultraspherical polynomials $P_n^{(\nu)}$, $n \geq 0$, we get the following integral formula
\begin{align*} 
\int_{-1}^1 \frac{P_{n}^{(\nu)}(x) (1-x^2)^{\nu-\frac{1}{2}} }{P_{n}^{(\nu)}(1) (1-x)^{\eps} }  dx &= \frac{(-1)^n \Gamma(\nu + \frac{1}{2})}{2^n \Gamma(n+\nu+\frac{1}{2})}
\int_{-1}^1 \left( \frac{d}{dx}\right)^n \left[(1-x^2)^{\nu - \frac{1}{2} + n}\right] \frac{1}{(1-x)^{\eps}} dx. 
 \end{align*}
Integration by parts of the right hand side yields (using, as in equation \eqref{equation-integralformula1}, \cite[(4.02)]{Ismail})
\begin{align} 
\int_{-1}^1 \frac{P_{n}^{(\nu)}(x) (1-x^2)^{\nu-\frac{1}{2}} }{P_{n}^{(\nu)}(1) (1-x)^{\eps} }  dx &= \frac{ \Gamma(n + \eps) \Gamma(\nu + \frac{1}{2})}{2^n \Gamma(n+\nu+\frac{1}{2})}
\int_{-1}^1 (1-x^2)^{\nu - \frac{1}{2} + n} \frac{1}{(1-x)^{n+\eps}} dx. \notag \\
&= \frac{ \Gamma(n + \eps) \Gamma(\nu + \frac{1}{2})}{2^n \Gamma(\eps)  \Gamma(n+\nu+\frac{1}{2})} \frac{2^{2\nu-\eps+n}\Gamma(n+\nu+\frac{1}{2}) \Gamma(\nu-\eps + \frac{1}{2})}{\Gamma(n+2\nu-\eps+1)} \notag \\
&= \frac{2^{2\nu-\eps}\Gamma(\nu+\frac{1}{2}) \Gamma(\nu-\eps + \frac{1}{2}) \Gamma(n + \eps)}{\Gamma(\eps)\Gamma(n+2\nu-\eps+1)}. \label{equation-integralformula2}
 \end{align}
Now, if $\nu > \frac{1}{2}$, we can choose $\eps = 1$ and formula \eqref{equation-seconddefinitionnumeratorpolynomials} for the numerator polynomials in combination 
with \eqref{equation-integralformula1} and \eqref{equation-integralformula2} gives
\begin{align*} 
\frac{Q_{n-1}^{(\nu)}(1)}{P_{n}^{(\nu)}(1)} &= \frac{1}{\beta_0} \int_{-1}^1 \left(1 - \frac{P_{n}^{(\nu)}(\xi)}{P_{n}^{(\nu)}(1)}\right) \frac{w_\nu(\xi)}{1-\xi} d\xi \\
&= \frac{2^{2\nu-1}\Gamma(\nu+\frac{1}{2})\Gamma(\nu-\frac{1}{2})}{\beta_0 \Gamma(2\nu)} - \frac{2^{2\nu-1}\Gamma(\nu+\frac{1}{2})\Gamma(\nu-\frac{1}{2}) \Gamma(n + 1)}{\beta_0 \Gamma(n+2\nu)} \\
&= \frac{2 \nu}{ 2 \nu -1} \left(1 - \frac{\Gamma(2 \nu) \Gamma(n+1)}{\Gamma(n+2\nu)} \right).
\end{align*}
On the other hand, if $- \frac{1}{2} < \nu \leq \frac{1}{2}$, we choose $\eps < \nu + \frac{1}{2}$ and get with \eqref{equation-integralformula1} and \eqref{equation-integralformula2}:
\begin{align*} 
\frac{Q_{n-1}^{(\nu)}(1)}{P_{n}^{(\nu)}(1)} & \geq \frac{1}{\beta_0} \int_{-1}^1 \left(1 - \frac{P_{n}^{(\nu)}(\xi)}{P_{n}^{(\nu)}(1)}\right) \frac{w_\nu(\xi)}{(1-\xi)^\eps} d\xi \\
&= \frac{2^{2\nu-\eps}\Gamma(\nu+\frac{1}{2})\Gamma(\nu+\frac{1}{2}-\eps)}{\beta_0 \Gamma(2\nu-\eps+1)}\left(1 - \frac{\Gamma(2 \nu-\eps+1) \Gamma(n+\eps)}{\Gamma(\eps) \Gamma(n+2\nu-\eps+1)} \right)\\
& \geq  \frac{\Gamma(\nu+\frac{1}{2}-\eps)}{2 \beta_0} \left(1 - \frac{\Gamma(n+\eps)}{\Gamma(n+2\nu-\eps+1)} \right).
\end{align*}
Therefore, we can find an $n_\eps \in \Nn$ such that
\begin{align*} 
\frac{Q_{n-1}^{(\nu)}(1)}{P_{n}^{(\nu)}(1)} & \geq \frac{\Gamma(\nu+\frac{1}{2}-\eps)}{4\beta_0} \qquad \text{for all $n \geq n_\eps$.}
\end{align*}
Since $\eps$ can be chosen arbitrarily close to $\nu + \frac{1}{2}$, the term on the right hand side can be arbitrarily large. Hence, in this case the
sequence $Q_{n-1}^{(\nu)}(1)/P_{n}^{(\nu)}(1)$, $n \in \Nn$, diverges. The formulas for the constant $L_1$ follow from the definition $L_1 = \lim_{n \to \infty} \frac{P_n^{(\nu)}(1)}{Q_{n-1}^{(\nu)}(1)}$.
\end{proof}

\begin{Lem} \label{Lemma-codilatedbounded}
For $\nu > \frac{1}{2}$, the polynomials $P_n^{(\nu)*}(x)/P_n^{(\nu)*}(1)$ are uniformly bounded on $[-1,1]$ by
\begin{equation}
\left| \frac{P_n^{(\nu)*}(x)}{P_n^{(\nu)*}(1)} \right| \leq \left\{ \begin{array}{cl} 1 & \qquad \text{if \quad $0 \leq \lambda \leq 1$,} \\ \frac{  |2 \nu (2 \lambda - 1) - \lambda|}{2 \nu - \lambda} & \qquad 
\text{if \quad $1 < \lambda < 2\nu$ \; or \; $\lambda < 0$.} \end{array}\right.
\end{equation}
\end{Lem}

\begin{proof}
For the ultraspherical polynomials with the parameter $\nu \geq 0$, it is well-known (see \cite{Szwarc2005} and the references therein) that $|P_n^{(\nu)}(x)| \leq P_n^{(\nu)}(1)$ and $|Q_n^{(\nu)}(x)| \leq Q_n^{(\nu)}(1)$ holds for all $x \in [-1,1]$, $n \in \Nn$. In the case $0 \leq \lambda \leq 1$, 
we get therefore by formula \eqref{equation-characteristiccodilated} the upper bound
\begin{equation*} 
\left| \frac{P_n^{(\nu)*}(x)}{P_n^{(\nu)*}(1)} \right|  
= \left| \frac{\lambda P_n^{(\nu)}(x) + (1-\lambda) x Q_{n-1}^{(\nu)}(x)}{\lambda P_n^{(\nu)}(1) + (1-\lambda) Q_{n-1}^{(\nu)}(1)} \right| \leq 
\frac{\lambda P_n^{(\nu)}(1) + (1-\lambda) Q_{n-1}^{(\nu)}(1)}{\lambda P_n^{(\nu)}(1) + (1-\lambda) Q_{n-1}^{(\nu)}(1)} = 1.
\end{equation*}
Similarly, we get for $1 < \lambda < 2 \nu$ or $\lambda < 0$:
\begin{align*} 
\left| \frac{P_n^{(\nu)*}(x)}{P_n^{(\nu)*}(1)} \right|  
& = \left| \frac{ P_n^{(\nu)}(x) + \frac{1-\lambda}{\lambda} Q_{n-1}^{(\nu)}(x)}{P_n^{(\nu)}(1) + \frac{1-\lambda}{\lambda} Q_{n-1}^{(\nu)}(1)} \right|
\leq \left| \frac{ P_n^{(\nu)}(1) - \frac{1-\lambda}{\lambda} Q_{n-1}^{(\nu)}(1)}{P_n^{(\nu)}(1) + \frac{1-\lambda}{\lambda} Q_{n-1}^{(\nu)}(1)} \right| \\
& \leq \left| \frac{ L_1 - \frac{1-\lambda}{\lambda} }{L_1 + \frac{1-\lambda}{\lambda} } \right| = 
\left| \frac{ \frac{2\nu-1}{2\nu} - \frac{1-\lambda}{\lambda} }{\frac{2\nu-1}{2\nu} + \frac{1-\lambda}{\lambda} } \right| = \frac{|2\nu(2\lambda-1) - \lambda|}{2\nu-\lambda}.
\end{align*}
In the last inequality, we used the fact that $L_1 = \frac{2\nu-1}{2\nu} =\underset{n \to \infty}{\lim} \frac{P_{n}^{(\nu)}(1)}{Q_{n-1}^{(\nu)}(1)} \leq \frac{P_{n}^{(\nu)}(1)}{Q_{n-1}^{(\nu)}(1)}$ holds 
for all $n \in \Nn$ (see formula \eqref{equation-chaininequality} in the proof of Lemma \ref{Lemma-zeroscodilated}).
\end{proof}

The weight function $w_\nu^{(1)}$ of the numerator polynomials $Q_n^{(\nu)}$ is supported on the interval $[-1,1]$ and can be stated explicitely as (see \cite[formulas (28) and (106)]{Grosjean1986})
\begin{align} \label{equation-formulaweightnumerator}
w_\nu^{(1)}(x) &:= \Phi_\nu(x) (1-x^2)^{\nu-\frac{1}{2}}, \quad \Phi_\nu(x):= \frac{1}{q_\nu(x)^2+\frac{\pi^2}{\beta_0^2}w_\nu(x)^2}, \\
q_{\nu}(x) &= \int_{0}^{x} \frac{2 \nu \, w_\nu(x) }{(1-t^2)^{\nu+\frac{1}{2}}} dt. \label{equation-formulasubweightnumerator}
\end{align}
We will soon see that for $\lambda < 2\nu$ and $x \in ]-1,1[$, the normalized co-dilated ultraspherical polynomials $P_n^{(\nu)*}(x)/P_n^{(\nu)*}(1)$ converge pointwise to zero.
The proof is based on the fact that for $\nu > \frac{1}{2}$ the weight function $w_{\nu}^{(1)}$ is a generalized Jacobi weight (see \cite[Definition 9.28]{Nevai}), i.e. $w_{\nu}^{(1)}$ is 
of the form \eqref{equation-formulaweightnumerator} with a continuous 
and strictly positive function $\Phi_\nu(x)$ on $[-1,1]$ whose modulus of continuity $\omega$ satisfies $\int_0^1 \delta^{-1} \omega(\Phi_\nu,\delta) d \delta < \infty$.  

\begin{Lem} 
For $\nu > \frac{1}{2}$, the weight function $w_{\nu}^{(1)}$ is a generalized Jacobi weight. 
\end{Lem}

\begin{proof}
We show that for $\nu > \frac{1}{2}$ the function $\Phi_\nu$ is strictly positive and Hölder-continuous on $[-1,1]$. Then, it follows immediately that $\int_0^1 \delta^{-1} \omega(\Phi_\nu,\delta) d \delta < \infty$ holds and, thus, that $w_\nu^{(1)}$ is a generalized Jacobi weight. Clearly, the weight function $w_\nu(x) = (1-x^2)^{\nu - \frac{1}{2}}$ of the ultraspherical polynomials is Hölder-continuous on $[-1,1]$ with exponent $\min \{ 1,\nu-\frac{1}{2}\}$. The function $q_\nu$ on the other hand is continuously differentiable on the open interval $(-1,1)$. So, to complete the proof it remains to show that $q_\nu$ satisfies a Hölder-condition and is nonzero at $x = 1$ and $x = -1$. Because of the symmetry of the weight function $w_{\nu}^{(1)}$, we have to study the behavior of $q_\nu(x)$ only at $x = 1$. To investigate the integral formula, we proceed similar as Szegö in \cite[Section 4.62]{Szegö} for the Jacobi polynomials of the second kind. We expand the factor $(1+t)^{\nu + \frac{1}{2}}$ in the integral formula
\eqref{equation-formulasubweightnumerator} of $q_\nu(x)$ in a power series in the variable $(1-t)$. Then, if $\nu > \frac{1}{2}$ and $\nu - \frac{1}{2}$ is not an integer we obtain
\[ q_\nu(x) = C (1-x^2)^{\nu - \frac{1}{2}} + M_1\left(\frac{1-x}{2}\right) (1+x)^{\nu-\frac{1}{2}},\]
with a power series $M_1(y)$ convergent for $|y|< 1$ and $M_1(0) \neq 0$. Thus, in this case the function $q_\nu$ is nonzero at $x=1$ and satisfies a Hölder-condition with exponent $\min\{\nu-\frac{1}{2},1\}$. If $n = \nu-\frac{1}{2} \geq 1$ is an integer, we get in the integrand of \eqref{equation-formulasubweightnumerator} the power series expansion (see \cite[p. 76]{Szegö}) 
\[\frac{1}{(1-t^2)^{n+1}} = \frac{\left( 1-\frac{1-t}{2} \right)^{-n-1}}{2^{n+1}(1-t)^{n+1}} = \frac{1}{2^{n+1}(1-t)^{n+1}} \left( 1 + \cdots + \binom{2n}{n}\left(\frac{1-t}{2}\right)^{n} + \cdots \right).\]
Integrating with respect to $t$ yields a logarithmic term and a power series $M_2(y)$ with $M_2(0) \neq 0$ converging for $|y|< 1$ such that 
\[ q_\nu(x) = C (1-x^2)^{n} \log \left( \frac{1}{1-x}\right) + M_2 \left(\frac{1-x}{2}\right) (1+x)^{n}.\]
Thus, in this case $q_\nu(x)$ is Lipschitz-continuous at $x = 1$ if $n \geq 2$, and Hölder-continuous with an arbitrary coefficient $0 < \alpha < 1$ if $n= 1$. 
\end{proof}

\begin{Thm} \label{Theorem-convergenceultraspherical} For $\nu > \frac{1}{2}$ and $\lambda < 2\nu$, the co-dilated ultraspherical polynomials $P_n^{(\nu)*}$ satisfy the estimate
\[ \left| \frac{P_n^{(\nu)*}(x)}{P_n^{(\nu)*}(1)}\right| 
\leq \frac{ 1+|\lambda|}{(2\nu-\lambda)} \frac{C_{\nu}}{(\sqrt{1-x^2} \, n + 1)^{\nu}} \]
with a constant $C_\nu$ independent of $n$, $x$ and $\lambda$. For the respective residual polynomials $r_n^{(\nu)*}(y) = \frac{P_n^{(\nu)*}(1 - 2y)}{P_n^{(\nu)*}(1)}$ on $[0,1]$ and 
the modulus of convergence $\eps_\nu^S(n)$, we get
\[  \left| r_n^{(\nu)*}(y)\right| \leq    \frac{1+|\lambda|}{(2\nu-\lambda)} \frac{ C_\nu}{ (2 \sqrt{y(1-y)}\,n+1)^\nu},\qquad \eps_\nu^S(n) \leq \frac{1+|\lambda|}{(2\nu-\lambda)} \frac{C_\nu }{(2n)^\nu}.\]
\end{Thm}

\begin{proof}
Since both, $w_\nu$ and $w_\nu^{(1)}$, are generalized Jacobi weights, we get for $P_n^{(\nu)}$ and $Q_n^{(\nu)}$ the uniform bounds (see \cite[Lemma 1.3]{Badkov1976} or \cite[Lemma 9.29]{Nevai}): 
\begin{align*}
\left(\sqrt{1-x^2} + \frac{1}{n}\right)^{\nu} \frac{|P_n^{(\nu)}(x)|}{\|P_n^{(\nu)}\|_{w_\nu}} \leq \left(\sqrt{1-x} + \frac{1}{n}\right)^{\nu} \left(\sqrt{1+x} + \frac{1}{n}\right)^{\nu} \frac{|P_n^{(\nu)}(x)|}{\|P_n^{(\nu)}\|_{w_\nu}} \leq C_1, \\
\left(\sqrt{1-x^2} + \frac{1}{n}\right)^{\nu} \frac{|Q_{n}^{(\nu)}(x)|}{\|Q_n^{(\nu)}\|_{w_\nu^{(1)}}} \leq \left(\sqrt{1-x} + \frac{1}{n}\right)^{\nu} \left(\sqrt{1+x} + \frac{1}{n}\right)^{\nu} \frac{|Q_{n}^{(\nu)}(x)|}{\|Q_n^{(\nu)}\|_{w_\nu^{(1)}}} \leq C_2,
\end{align*}
with constants $C_1$ and $C_2$ independent of $x \in [-1,1]$ and $n \in \Nn$. Due to the particular normalization \eqref{equation-formulaweightnumerator} of the weight function
$w_\nu^{(1)}$, the weighted $L^2$-norms of the monic polynomials $P_n^{(\nu)}$ and $Q_{n-1}^{(\nu)}$ coincide (see \cite[Section 2]{Grosjean1986}), i.e. 
$\| P_n^{(\nu)}\|_{w_\nu} = \| Q_{n-1}^{(\nu)}\|_{w_\nu^{(1)}}$. This yields the following estimate for the co-dilated polynomials: 
\begin{align*}
\left(\sqrt{1-x^2} + \frac{1}{n}\right)^{\nu} \frac{|P_n^{(\nu)*}(x)|}{\| P_n^{(\nu)}(x)\|_{w_\nu}} &= \left(\sqrt{1-x^2} + \frac{1}{n}\right)^{\nu} \left|  \frac{ \lambda P_n^{(\nu)}(x)}{\| P_n^{(\nu)}(x)\|_{w_\nu}} + \frac{ (1-\lambda) x Q_{n-1}^{(\nu)}(x)}{\| P_n^{(\nu)}(x)\|_{w_\nu}} \right| \\
&\leq |\lambda| C_1 + |1-\lambda| C_2  \leq  (1+|\lambda|) ( C_1 + C_2).
\end{align*}
By Lemma \ref{Lemma-asymptoticone} and Lemma \ref{Lemma-criticalconstantultraspherical}, we have $\frac{P_n^{(\nu)}(1)}{P_n^{(\nu)*}(1)} \leq \frac{L_1}{1+(L_1-1) \lambda} = \frac{2\nu - 1}{2\nu-\lambda}$.
Therefore,
\begin{equation} \label{equation-intermediate1}
 \left(\sqrt{1-x^2} + \frac{1}{n}\right)^{\nu} \frac{|P_n^{(\nu)*}(x)|}{P_n^{(\nu)*}(1)} \leq \frac{\|P_n^{(\nu)}\|_{w_{\nu}} }{P_n^{(\nu)}(1)} 
\frac{2\nu-1}{2\nu - \lambda}  (1+|\lambda|) (C_1 + C_2).
\end{equation}
For the ultraspherical polynomials, the quotient $P_n^{(\nu)}(1)^2 / \|P_n^{(\nu)}\|_{w_{\nu}}^2$ can be computed explicitely as (for the formulas, see \cite[p. 30]{Gautschi})
\begin{equation} \label{equation-quotientofdifferentnormalizations}
\frac{P_n^{(\nu)}(1)^2}{\|P_n^{(\nu)}\|_{w_{\nu}}^2} = \frac{\Gamma(n+2 \nu)}{\Gamma(n+1)} \frac{n+\nu}{2^{2\nu-1}\Gamma(\nu+\frac{1}{2})^2}.
\end{equation}
Using two inequalities related to the formula of Gosper for the Gamma function $\Gamma(x)$ (see \cite[Theorem 1]{Mortici2011}), we get for $\nu > \frac{1}{2}$ the following lower bound:
\begin{align*} \frac{P_n^{(\nu)}(1)^2}{\|P_n^{(\nu)}\|_{w_{\nu}}^2} &\geq \frac{\left(\frac{n+2\nu-1}{e}\right)^{n+2\nu-1} \sqrt{2n+4\nu-2+\frac{1}{3}}}{\left(\frac{n}{e}\right)^n \sqrt{2n+1}}  \frac{n+\nu}{2^{2\nu-1}\Gamma(\nu+\frac{1}{2})^2} \geq \frac{n^{2\nu}}{ (2e)^{2\nu-1} \Gamma(\nu + \frac{1}{2})^2}. 
\end{align*}

Now, including this inequality in the estimate \eqref{equation-intermediate1}, we get a constant $C_\nu$ independent of 
$ \lambda < 2\nu$, $x \in [-1,1]$ and $n \in \Nn$ such that
\[ \left(\sqrt{1-x^2} + \frac{1}{n}\right)^{\nu} \left| \frac{P_n^{(\nu)*}(x)}{P_n^{(\nu)*}(1)}\right| 
\leq  \frac{(1+|\lambda|) }{2\nu-\lambda}   \frac{C_\nu}{ n^{\nu}} . \]
The estimates for the residual polynomial $r_n^{(\nu)*}$ on $[0,1]$ and the modulus $\eps_\nu^S(n)$ follow directly from the respective definitions. 
\end{proof}

\begin{Rem}
Theorem \ref{Theorem-convergenceultraspherical} implies that for the semi-iterative algorithms based on the co-dilated ultraspherical polynomials $P_n^{(\nu)*}$ with parameter $\lambda < 2\nu$ 
the symmetric modulus of convergence is of order $\eps_{s}^S(n) = \Ord(n^{-s})$ for $0 < s \leq \nu$. Therefore, according to \cite[Theorem 4.1]{Hanke1991}), the co-dilated ultraspherical polynomials provide a semi-iterative method with optimal order $\Ord(n^{-s})$ of convergence if the solution $f$ is an element of $X_s$, $0<s \leq \nu$.  
\end{Rem}

Finally, for the co-dilated ultraspherical polynomials, we compute the coefficients $\mu_{n+1}$ in Algorithm \ref{algorithm-2} explicitly. To this end, we need first of all an explicit formula
for the quotient $P_n^{(\nu)}(1)/P_{n+1}^{(\nu)}(1)$. We obtain this quotient by using formula \eqref{equation-quotientofdifferentnormalizations} and the fact that $\|P_{n+1}^{(\nu)}\|_{w_\nu} = \sqrt{\beta_{n+1}}\|P_{n}^{(\nu)}\|_{w_\nu}$ holds for the monic polynomials $P_{n}^{(\nu)}$. Thus, we get
\begin{align*}
\frac{P_n^{(\nu)}(1)}{P_{n+1}^{(\nu)}(1)} &= \frac{P_n^{(\nu)}(1)}{\sqrt{\beta_{n+1}}\|P_{n}^{(\nu)}\|_{w_\nu}} \frac{\|P_{n+1}^{(\nu)}\|_{w_\nu}}{P_{n+1}^{(\nu)}(1)} 
= 2 \frac{n+\nu}{n+2\nu}. 
\end{align*}
Now, using formula \eqref{equation-quotientnumerator}, we get the coefficients $\mu_{n+1}$, $n \geq 1$, explicitly.
\begin{align} \label{equation-coefficientssemiiterativecodilatedultraspherical}
\mu_{n+1} = \frac{P_n^{(\nu)*}(1)}{P_{n+1}^{(\nu)*}(1)} &= \frac{P_n^{(\nu)}(1)}{P_{n+1}^{(\nu)}(1)}\frac{\lambda + (1-\lambda) \frac{Q_{n-1}^{(\nu)*}(1)}{P_{n}^{(\nu)*}(1)}}{\lambda + (1-\lambda) \frac{Q_{n}^{(\nu)*}(1)}{P_{n+1}^{(\nu)*}(1)}} \notag \\& = 2 \frac{n+\nu}{n+2\nu} \frac{\lambda + (1-\lambda) \frac{2 \nu}{ 2 \nu -1} \left(1 - \frac{\Gamma(2 \nu) \Gamma(n+1)}{\Gamma(n+2\nu)} \right)}{\lambda + (1-\lambda) \frac{2 \nu}{ 2 \nu -1} \left(1 - \frac{\Gamma(2 \nu) \Gamma(n+2)}{\Gamma(n+2\nu+1)} \right)}. 
\end{align}
With a simplified expression for $\mu_{n+1}$, the semi-iterative method based on the co-dilated ultraspherical polynomials is summarized in Algorithm \ref{algorithm-8}.

\begin{algorithm} 
\caption{Semi-iterative method based on co-dilated ultraspherical polynomials}
\label{algorithm-8}
\begin{algorithmic}[H]
\vspace{2mm}
\STATE $f_0 = 0$, $f_1 = 2 \; \omega A^* g $ \vspace{2mm}
\WHILE {(stopping criterion false)}  \vspace{2mm}
\STATE $\mu_{n+1} = 2 (n+\nu) \frac{(2\nu-\lambda) \Gamma(n+2\nu) + (\lambda-1) \Gamma(2\nu+1) \Gamma(n+1)}{(2\nu-\lambda) \Gamma(n+2\nu+1) + (\lambda-1) \Gamma(2\nu+1) \Gamma(n+2)} $ \\[2mm]
\STATE $f_{n+1} = f_n + (\mu_{n+1}-1) ( f_n - f_{n-1}) + 2 \mu_{n+1} \; \omega  A^*( g - A f_n)$\\[2mm]
\STATE $n \to n+1$
\ENDWHILE
\end{algorithmic}
\end{algorithm}

\section{Co-dilated $\nu$-methods} \label{section-comodified}

The $\nu$-methods correspond to Algorithm \ref{algorithm-2} with the recurrence coefficients $\alpha_n$, $\beta_n$ of the monic Jacobi polynomials $P_n^{(\nu - \frac{1}{2},-\frac{1}{2})}$ on $[-1,1]$. These particular orthogonal polynomials
are linked to the ultraspherical polynomials $P_{n}^{(\nu)}$ by the formula (see \cite[Theorem 4.1]{Szegö}, using the normalization of the monic polynomials) 
\[ P_n^{(\nu - \frac{1}{2},-\frac{1}{2})}(x) = 2^{n} P_{2n}^{(\nu)}\left(\sqrt{\frac{1+x}{2}}\right), \qquad x \in [-1,1]. \]
In other words, the polynomials $P_n^{(\nu - \frac{1}{2},-\frac{1}{2})}$ describe the positive part of the axisymmetric ultraspherical polynomials $P_{2n}^{(\nu)}$.
Thus, for the asymmetric residual polynomials $\resa_n^{(\nu)}$ of the $\nu$-methods, we have
\[ \resa_n^{(\nu)}(y) := \frac{P_n^{(\nu - \frac{1}{2},-\frac{1}{2})}(1-2y)}{P_n^{(\nu - \frac{1}{2},-\frac{1}{2})}(1)} 
= \frac{P_{2n}^{(\nu)}\left(\sqrt{1-y}\right)}{P_{2n}^{(\nu)}(1)}, \qquad y \in [0,1]. \]

Compared to the semi-iterative methods based on 
the ultraspherical polynomials, the $\nu$-methods have the advantage to converge if $\omega \|A^* A\| = 1$. A similar approach for arbitrary symmetric orthogonal polynomials leads us now to
semi-iterative methods that generalize the $\nu$-methods.

In general, if $P_{2n}(x)$ is an arbitrary even polynomial of degree $2n$ on the interval $[-1,1]$, then $P_{2n}(\sqrt{1-y})$ defines a polynomial of degree $n$ in the variable $y$ on the interval $[0,1]$. In this case, we can define asymmetric residual polynomials $\resa_n$ by
\begin{equation} \label{equation-residualpolynomialsalternative} \resa_n(y) := \frac{P_{2n}(\sqrt{1-y})}{P_{2n}(1)}, \qquad y \in [0,1].\end{equation}
Moreover, if the symmetric polynomials $P_n$ satisfy the three-term recurrence formula \eqref{equation-3termrecurrence}, we can deduce directly a three-term recurrence relation for the 
residual polynomials $\resa_n$. Applying the relation \eqref{equation-3termrecurrence} twice, we get first for the even polynomials $P_{2n}$ the recurrence
\begin{align} \label{equation-3termrecurrenceasymmetric} 
& P_{2n+2}(x) = (x^2 - \beta_{2n} - \beta_{2n+1}) P_{2n}(x) - \beta_{2n}\beta_{2n-1} P_{2n-2}(x), \quad n \geq 1, \notag \\ & P_0(x) = 1, \qquad P_2(x) = x^2-\beta_1. 
\end{align}
Inserting \eqref{equation-3termrecurrenceasymmetric} in the definition \eqref{equation-residualpolynomialsalternative}, yields the following recursion formula
for the residual polynomial $\resa_{n}(y)$ on $[0,1]$:
\begin{align} \label{equation-3termrecurrenceresidualasymmetric}
 & \resa_{n+1}(y) = (1 - y - \beta_{2n}-\beta_{2n+1}) \frac{  P_{2n}(1) }{P_{2n+2}(1)} \resa_n(y) - \beta_{2n} \beta_{2n-1} \frac{  P_{2n-2}(1) }{P_{2n+2}(1)} \resa_{n-1}(y), \notag \\ 
 & \resa_0(y) = 1, \qquad \resa_1(y) = \frac{1-\beta_1 - y}{1-\beta_1}. 
\end{align}
By the formula \eqref{equation-3termrecurrenceasymmetric}, also the factors $\mua_{n+1} = \frac{P_{2n}(1)}{P_{2n+2}(1)}$ can be computed recursively. This results in the following
semi-iterative Algorithm \ref{algorithm-5}.
\begin{algorithm} 
\caption{Semi-iterative method based on the asymmetric residual polynomials $\resa_n$}
\label{algorithm-5}
\begin{algorithmic}[H]
\STATE $\mua_1 = \frac{1}{1 - \beta_1}$
\STATE $f_0 = 0$, $f_1 = \mua_1 \; \omega A^* g $
\WHILE {(stopping criterion false)} 
\STATE $\mua_{n+1} = \frac{1}{1-\beta_{2n}-\beta_{2n+1} - \beta_{2n}\beta_{2n-1} \mua_n }$ \\[2mm]
\STATE $f_{n+1} = f_n + ((1-\beta_{2n}-\beta_{2n+1}) \mua_{n+1}-1) ( f_n - f_{n-1}) + \mua_{n+1} \; \omega  A^*( g - A f_n)$\\[2mm]
\STATE $n \to n+1$
\ENDWHILE
\end{algorithmic}
\end{algorithm}

In the light of \eqref{equation-residualpolynomialsalternative}, we can introduce asymmetric residual polynomials also for the co-dilated orthogonal polynomials $P_n^*$ by setting
\begin{equation} \label{equation-residualpolynomialsalternative2} 
\resa_n^*(y) := \frac{P_{2n}^*(\sqrt{1-y})}{P_{2n}^*(1)}, \qquad y \in [0,1].
\end{equation}
In view of the recurrence relation \eqref{equation-3termrecurrencedilated} of the co-dilated polynomials $P_n^*$, the residual polynomials
$\resa_n^*$ satisfy the same recurrence relation \eqref{equation-3termrecurrenceresidualasymmetric} as the polynomials $\resa_n$ except that the two coefficients including $\beta_m$ are altered. Families of orthogonal polynomials in which more than one coefficient is altered are known as co-modified orthogonal polynomials. As the co-dilated polynomials, they are well
studied in the literature, see \cite{DiniMaroniRonveaux1989, MarcellanDehesaRonveaux1990, RonveauxBelmehdiDiniMaroni1990}. 

From Theorem \ref{Theorem-largerzeros} and the Lemmas \ref{Lemma-zeroscodilated} and \ref{Lemma-asymptoticone}, we can moreover deduce the following results about the zeros of the polynomials
$\resa_n^*$. The statements follow directly from the relation \eqref{equation-residualpolynomialsalternative2} of the polynomials $\resa_n^*$ to the polynomials $P_{2n}^*$. 

\begin{Cor}
The smallest zero of $\resa_n^*$ is a decreasing function of the dilation parameter $\lambda$.
All zeros of $\resa_n^*(y)$, $n \in \Nn$, are in the interior of $[0,1]$ if and only if $\lambda < \frac{1}{1-L_m}$.
\end{Cor}

Finally, for $m = 1$, we consider in more detail the asymmetric residual polynomials $\resa_n^{(\nu)*}(y) = P_{2n}^{(\nu)*}(\sqrt{1-y})/P_{2n}^{(\nu)*}(1)$ linked 
to the co-dilated ultraspherical polynomials. As a consequence of Theorem \ref{Theorem-convergenceultraspherical}, we get the following estimates for $\resa_n^{(\nu)*}(y)$. 

\begin{Cor} \label{Corollary-convergencenonsymmetric} For $m = 1$, $\nu > \frac{1}{2}$ and $\lambda < 2\nu$, the residual polynomials 
$\resa_n^{(\nu)*}(y)$ on $[0,1]$ and the modulus of convergence $\eps_\nu(n)$ are bounded by
\[  \left| \resa_n^{(\nu)*}(y) \right| \leq    \frac{1+|\lambda|}{(2\nu-\lambda)} \frac{C_\nu}{ (2 \sqrt{y}\,n+1)^\nu}, \qquad \eps_\nu(n) \leq \frac{1+|\lambda|}{(2\nu-\lambda)} \frac{C_\nu}{ (2 n)^\nu}.\]
The constant $C_\nu$ is independent of $n$, $y$ and $\lambda$. 
\end{Cor}

\begin{Rem}
For $\lambda = 1$, the result of Corollary \ref{Corollary-convergencenonsymmetric} corresponds precisely to the well-known convergence result of the $\nu$-methods, see \cite[Theorem 6.12]{EnglHankeNeubauer}. 
In Corollary \ref{Corollary-convergencenonsymmetric}, the residual polynomials $\resa_n^{(\nu)*}$ converge pointwise to zero at $y = 1$. Thus, compared to 
the symmetric polynomials $r_n^{(\nu)*}$ of Theorem \ref{Theorem-convergenceultraspherical}, the residual polynomials $\resa_n^{(\nu)*}$ define semi-iterative methods that converge also to zero
if $\omega \|A^* A\| = 1$.
\end{Rem}
\begin{Rem}
The convergence orders $\eps_\nu(n) = \Ord(n^{-\nu})$ obtained in Corollary \ref{Corollary-convergencenonsymmetric} are substantial for the usage of the co-dilated $\nu$-methods as regularization methods. In particular, \cite[Theorem 6.11]{EnglHankeNeubauer} implies that the co-dilated $\nu$-method with $\lambda < 2 \nu$ based on the residual polynomials $\resa_n^{(\nu)*}$ is a regularization method of optimal order for $f \in X_s$ with $0 < s \leq \nu-1$ if 
the iteration $f_n$ is stopped according to the discrepancy principle, i.e. if $\|A f_n - g\| < \tau \eps$. Here, $\eps$ denotes the noise level of the data and the parameter $\tau$ is chosen larger
than the uniform bound $\sup_{y \in [0,1]}|\resa_n^{(\nu)*}(y)|$ given in Lemma \ref{Lemma-codilatedbounded}.
Using a generalized discrepancy principle as stopping rule, as described in \cite[Algorithm 6.17]{EnglHankeNeubauer} and \cite{HankeEngl1994}, the co-dilated $\nu$-methods 
even provide an order optimal regularization method for $0 < s \leq \nu$ (see \cite[Theorem 6.18]{EnglHankeNeubauer}). 
\end{Rem}

In Algorithm \ref{algorithm-5}, the coefficients $\mua_{n+1}$ for the co-dilated ultraspherical polynomials are given explicitly as
$\mua_{n+1} = \mu_{2n+1} \mu_{2n+2}$, the factors $\mu_{n+1}$ given in \eqref{equation-coefficientssemiiterativecodilatedultraspherical}. 
With the recursion coefficients $\beta_n$ of the ultraspherical polynomials given in \eqref{equation-recursioncoefficientsultraspherical}, we summarize the co-dilated $\nu$-methods in
Algorithm \ref{algorithm-4}.

\begin{algorithm} 
\caption{Co-dilated $\nu$-methods}
\label{algorithm-4}

\begin{algorithmic}[H]
\vspace{2mm}
\STATE $f_0 = 0$, $f_1 = \frac{2\nu+2}{2\nu+2-\lambda} \; \omega A^* g $ \vspace{2mm}
\WHILE {(stopping criterion false)}  \vspace{2mm}
\STATE $\mua_{n+1} = 4 (2n+\nu) (2n+\nu+1) \frac{(2\nu-\lambda) \Gamma(2n+2\nu) + (\lambda-1) \Gamma(2\nu+1) \Gamma(2n+1)}{(2\nu-\lambda) \Gamma(2n+2\nu+2) + (\lambda-1) \Gamma(2\nu+1) \Gamma(2n+3)}$ \\[2mm]
\STATE $f_{n+1} = f_n + \left(\!\left( 1 - \frac{4 n^2+4 \nu n + \nu -1}{2(2n+\nu+1)(2n+\nu-1)} \!\right) \! \mua_{n+1}-1 \! \right)\! ( f_n - f_{n-1}) + \mua_{n+1} \; \omega  A^*( g - A f_n)$\\[2mm]
\STATE $n \to n+1$
\ENDWHILE
\end{algorithmic}
\end{algorithm}

\section{Adaptive choice of the dilation parameter $\lambda$ for the co-dilated $1$-method}

In the algorithms of the last sections it is a priori not clear how the dilation parameter $\lambda$ has to be chosen. In the following, we provide for the co-dilated $1$-method a simple 
adaptive scheme that computes for every step $n$ of the iteration an optimal $\lambda$ such that the error $\|A f_n - g\|$ is minimized.
For $\nu = 1$, the ultraspherical polynomials $P_n^{(1)}$ coincide with the Chebyshev polynomials $U_n$ of the second kind. This yields in Algorithm \ref{algorithm-4} the coefficients
\[ \mua_{n+1} = \frac{U_{2n}^*(1)}{U_{2n+2}^*(1)} = 4 \frac{(2-\lambda)2n + \lambda}{(2-\lambda)2n + 4 - \lambda},\]
resulting in the iterative scheme
\begin{align} \label{equation-generalizednemirovskiipolyak}
f_{n+1} &= f_n + \frac{(2-\lambda)2n + 3 \lambda-4}{(2-\lambda)2n+4-\lambda}\;(f_n - f_{n-1}) + 4 \,\frac{(2-\lambda)2n+\lambda}{(2-\lambda) 2n + 4-\lambda}\, \omega \;A^*( g - A f_n), \notag \\
      f_1 &= \frac{4}{4-\lambda} \omega \; A^* g, \qquad f_0 = 0.
\end{align}
For $\lambda = 1$, this iteration is precisely the Chebyshev method of Nemirovskii and Polyak (see \cite[p. 150]{Rieder}). Due to the particular three-term recurrence formula \eqref{equation-3termrecurrencechebyshev} of the Chebyshev
polynomials $U_n$, it is possible to calculate the iterates $f_n$ for all different $\lambda$ at one stroke. Namely, by the last identity in equation \eqref{equation-definitioncodilatedchebyshev} the 
residual polynomials $\resa_n^{*}(y) := \resa_n^{(1)*}(y)$ can be written as
\begin{align} \label{equation-linearcombinationresidual}
\resa_n^{*}(y) &= \frac{U_{2n}^*(\sqrt{1-y})}{U_{2n}^*(1)} = \frac{ U_{2n}(\sqrt{1-y}) + \frac{1-\lambda}{4} U_{2n-2}(\sqrt{1-y})}{U_{2n}^*(1)} \notag \\
& = \frac{  U_{2n}(1)}{U_{2n}^*(1)} \frac{U_{2n}(\sqrt{1-y})}{U_{2n}(1)} + \frac{1-\lambda}{4} \frac{  U_{2n-2}(1)}{U_{2n}^*(1)} \frac{U_{2n-2}(\sqrt{1-y})}{U_{2n-2}(1)} \notag \\
&= \frac{ 2n+1}{(2-\lambda)2n+\lambda} \resa_n(y) +  \frac{(1-\lambda)(2n-1)}{(2-\lambda)2n+\lambda} \resa_{n-1}(y). 
\end{align}
In this way, every residual polynomial $\resa_n^{*}(y)$ can be computed as an affine combination of the residual polynomials $\resa_n$ and $\resa_{n-1}$. This enables us to introduce a low-cost
adaptive algorithm in which in every step $n$ the parameter $\lambda$ is chosen optimally. If $f_n^\lambda$ denotes the iterate in \eqref{equation-generalizednemirovskiipolyak} with respect to a fixed parameter $\lambda \in \Rr$, we have 
\[ \min_{\lambda \in \Rr} \|A f_n^\lambda - g\| = \min_{\gamma \in \Rr} \|A (f_n^1 - \gamma (f_n^1-f_{n-1}^1)) - g\|.\]
The minimum on the right hand side is obtained if the vector $ f_n^1 - \gamma (f_n^1- f_{n-1}^1)$ is orthogonal to $f_n^1- f_{n-1}^1$, i.e. if
\[ \arg \min_{\gamma \in \Rr} \|A (f_n^1 - \gamma (f_n^1-f_{n-1}^1)) - g\|= \frac{\langle f_n^1 , f_n^1-f_{n-1}^1 \rangle}{\|f_{n}^1-f_{n-1}^1\|^2}. \]
Thus, in view of \eqref{equation-linearcombinationresidual}, the optimal $\lambda$ after $n$ steps of the iteration \eqref{equation-generalizednemirovskiipolyak} is given by
\[ \lambda = 1 - \frac{(2n+1)\langle f_n^1 , f_n^1-f_{n-1}^1 \rangle}{(2n-1)(\|f_n^1 - f_{n-1}^1\|^2 - \langle f_n^1 , f_n^1-f_{n-1}^1 \rangle)}.\]
This simple idea is summarized in the adaptive co-dilated $1$-method formulated in Algorithm \ref{algorithm-7}. 
Here, the iteration is stopped according to the discrepancy principle if the minimal error $\min_{\lambda \in \Rr} \|A f_n^\lambda - g\|$ gets smaller than $\tau \eps$, where $\tau > 1$ and
$\eps$ describes the noise level of the data. 

\begin{algorithm} 
\caption{Adaptive co-dilated $1$-method}
\label{algorithm-7}

\begin{algorithmic}[H]
\vspace{2mm}
\STATE $f_0 = 0$, $f_1 = \frac{4}{3} \omega A^* g $
\STATE $v_0 = g$, $v_1 = g - A f_1$ 
\STATE $\gamma = \frac{\langle v_{1} , v_{1}-v_{0} \rangle}{\|v_{1}-v_{0}\|^2}$, $v_{\min} = v_1 - \gamma( v_1 - v_0 )$
\vspace{2mm}
\WHILE {$\|v_{\min}\| > \tau \eps $}  \vspace{2mm}
\STATE $f_{n+1} = f_n + \frac{2n-1}{2n+3} ( f_n - f_{n-1}) + 4 \omega \frac{2n+1}{2n+3} A^*v_n$\\[2mm]
\STATE $v_{n+1} = g - A f_{n+1}$\\[2mm]
\STATE $\gamma = \frac{\langle v_{n+1} , v_{n+1}-v_{n} \rangle}{\|v_{n+1}-v_{n}\|^2}$ \\[2mm]
\STATE $v_{\min} = v_{n+1} - \gamma (v_{n+1} - v_n)$\\[2mm]
\STATE $n \to n+1$
\ENDWHILE \\[2mm]
\STATE $\lambda = 1 - \frac{(2n+1) \gamma }{(2n-1)(1-\gamma)}$ \\[2mm]
\STATE $f_n = f_n - \gamma (f_n - f_{n-1})$
\end{algorithmic}
\end{algorithm}

\begin{Rem}
Similar adaptive schemes are in principle possible also for the other co-dilated $\nu$-methods. Taking two arbitrary real values $\lambda_1 \neq \lambda_2$, every iterate $f_n^\lambda$ can
be written as an affine combination of $f_n^{\lambda_1}$ and $f_n^{\lambda_2}$. Thus, in order to obtain all iterates $f_n^\lambda$, it suffices to compute two iterates 
$f_n^{\lambda_1}$ and $f_n^{\lambda_2}$. However, for general $\nu$ there exists no direct relation between $f_n^{\lambda_1}$ and $f_n^{\lambda_2}$ such as 
in the case of the Chebyshev polynomials. Therefore, for $\nu \neq 1$ the iterations in adaptive schemes like Algorithm \ref{algorithm-7} are twice as expensive
as in Algorithm \ref{algorithm-4}.
\end{Rem}

\section{Numerical tests}

In this final section, we compare the convergence behavior of the co-dilated $\nu$-methods with the original $\nu$-methods and the Landweber method. As a first and very simple test equation 
we consider 
\begin{equation} \label{equation-testequation1} 
A f = g_1^{\eps}, \qquad A = \diag \left(1,\frac{1}{2}, \cdots \frac{1}{N}\right), \quad g_1^{\eps} = e_N + \eps w,
\end{equation}
with $e_N = (0, \cdots, 0,1)$, $\eps > 0$ and $w$ a vector of normally distributed Gaussian white noise.  
To solve \eqref{equation-testequation1}, we use Algorithm \ref{algorithm-4}. We choose $\omega = 1$, $\eps = 0.01$, $N = 100$ and stop the iteration according to the discrepancy principle, if 
$\|A f_n - g_1^{\eps}\| < 4 \eps$. For $\nu = 1$ the number of necessary iterations depending on the dilation parameter $1 \leq \lambda \leq 2.2$ is given in Figure \ref{figure-1}.
For $\nu = 2$, the number of necessary iterations depending on the parameter $3.9 \leq \lambda \leq 4.05$ is depicted in Figure \ref{figure-2}. In both figures, the smallest zeros 
of the respective residual polynomials are plotted on the right hand side. 
 
\begin{figure}[H] \caption{Convergence of the co-dilated $1$-method (Algorithm \ref{algorithm-4} with $\nu = 1$) to solve \eqref{equation-testequation1} depending on the dilation 
parameter $\lambda$.} \label{figure-1} 
  \begin{minipage}{0.5\textwidth}
  \centering
  \includegraphics[angle=-90, width=\textwidth]{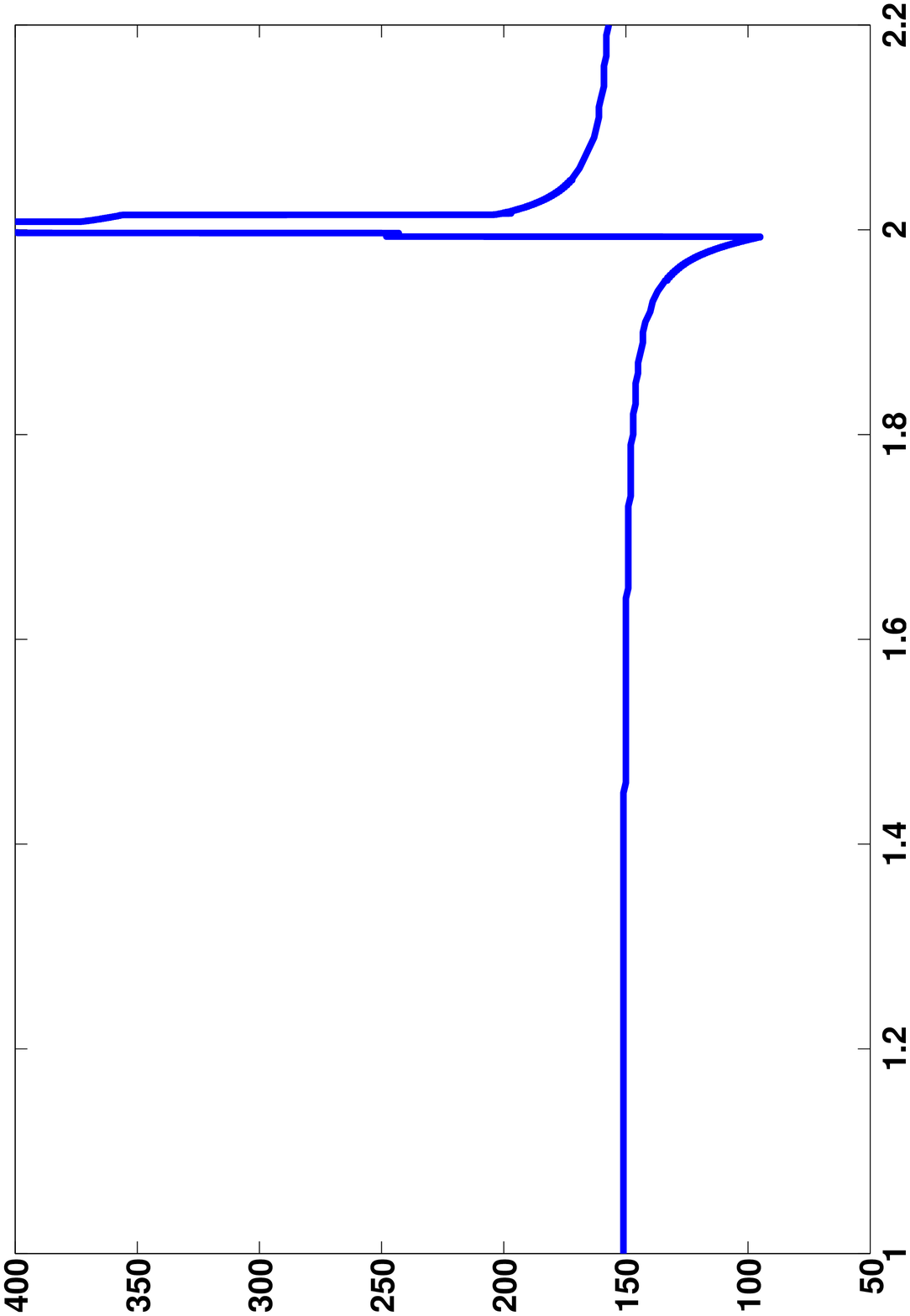}\\
  Number of iterations to solve \eqref{equation-testequation1} with Algorithm \ref{algorithm-4} with $\nu = 1$ depending on $1 \leq \lambda \leq 2.2$. 
  \end{minipage}\hfill
  \begin{minipage}{0.5\textwidth}
  \centering
  \includegraphics[angle=-90, width=\textwidth]{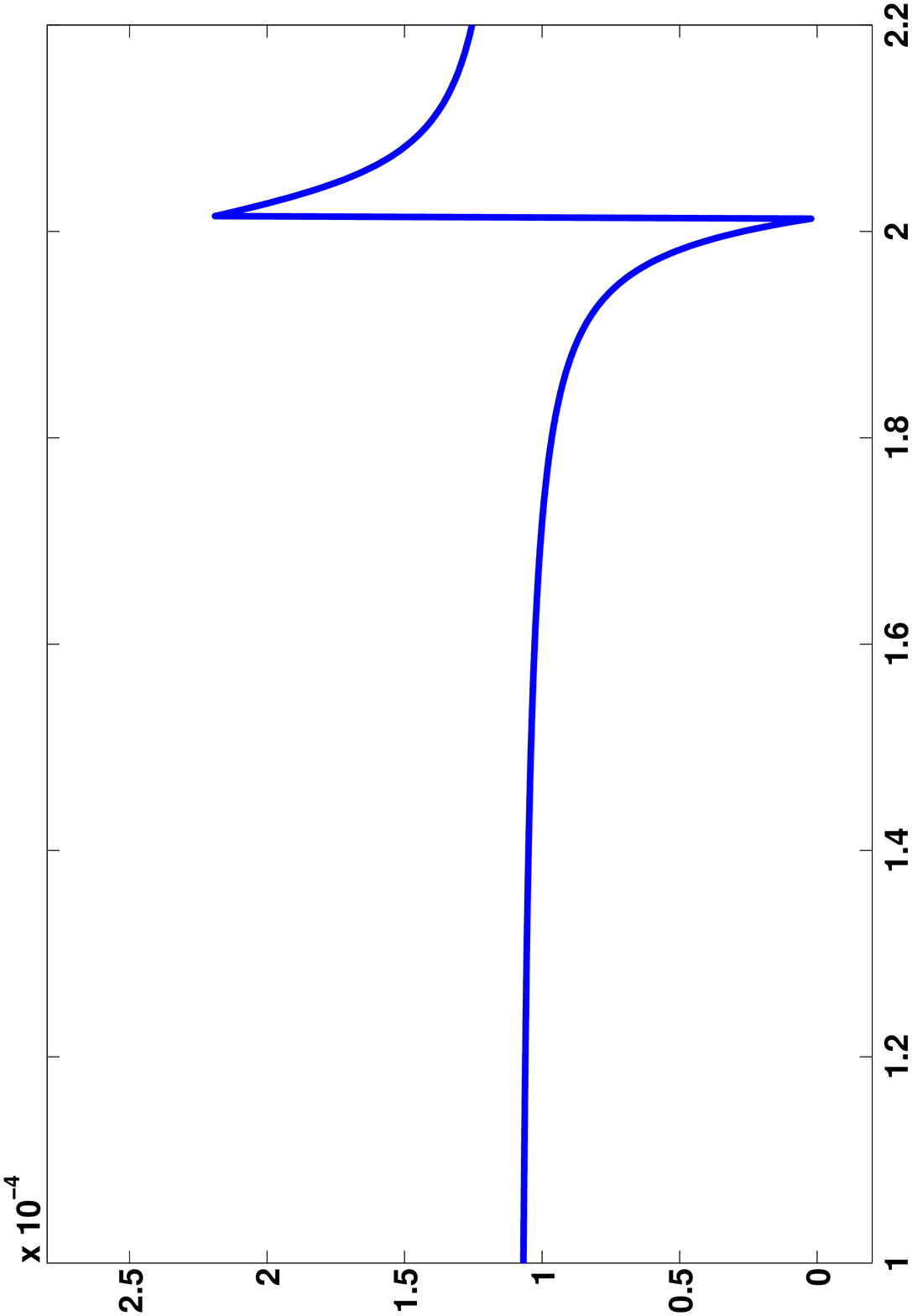}\\
  Smallest zero of the residual polynomial $\resa_{150}^{(1)*}(y)$ in the interval $[0,1]$ depending on $1 \leq \lambda \leq 2.2$.
  \end{minipage}
\end{figure}

\begin{figure}[H] \caption{Convergence of the co-dilated $2$-method (Algorithm \ref{algorithm-4} with $\nu = 2$) to solve \eqref{equation-testequation1} depending on the dilation parameter $\lambda$, .} \label{figure-2}
  \begin{minipage}{0.5\textwidth}
  \centering
  \includegraphics[angle=-90, width=\textwidth]{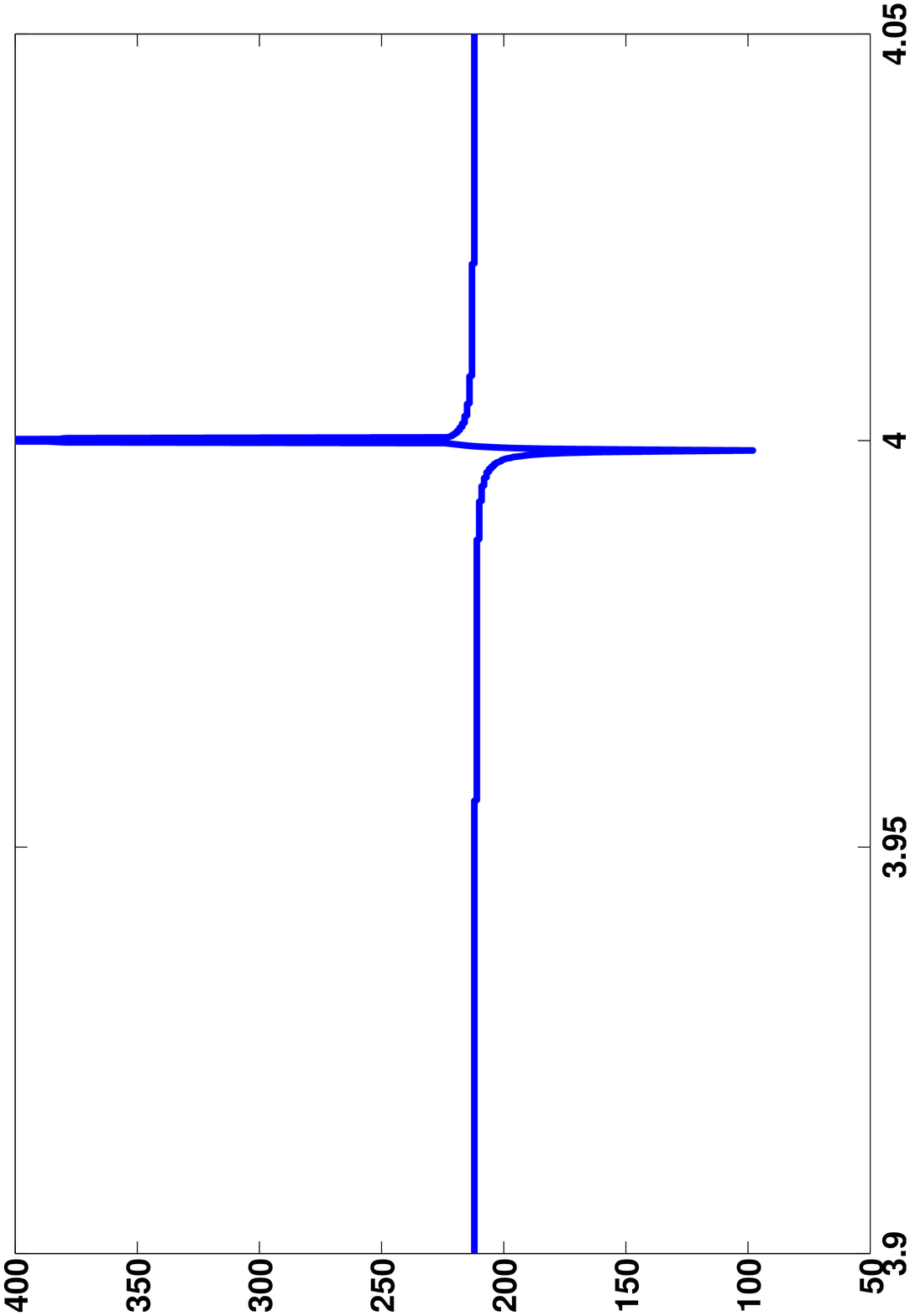}\\
  Number of iterations to solve \eqref{equation-testequation1} with Algorithm \ref{algorithm-4} with $\nu = 2$ depending on $3.9 \leq \lambda \leq 4.05$. 
  \end{minipage}\hfill
  \begin{minipage}{0.5\textwidth}
  \centering
  \includegraphics[angle=-90, width=\textwidth]{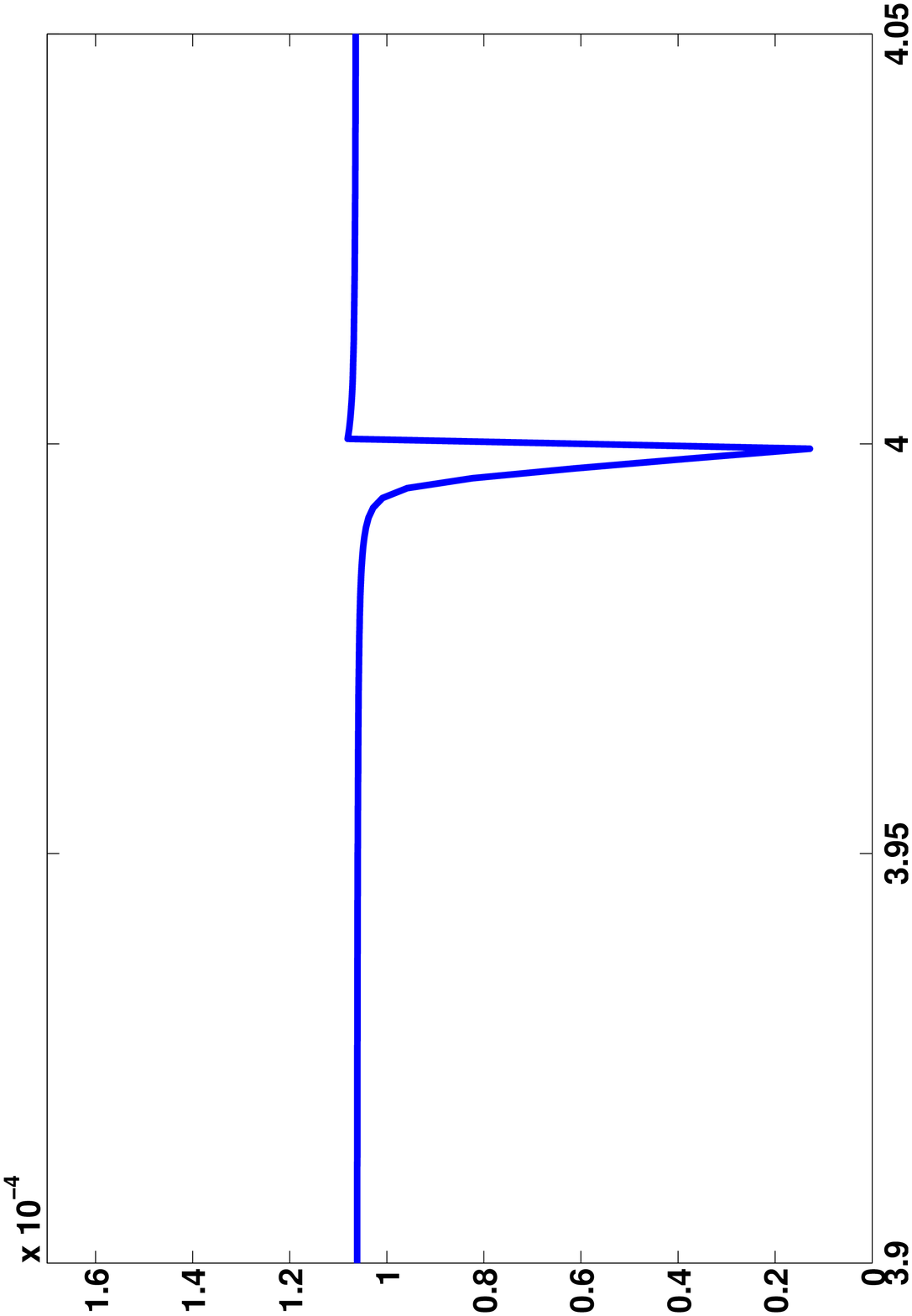}\\
  Smallest zero of the residual polynomial $\resa_{215}^{(2)*}(y)$ in the interval $[0,1]$ depending on $3.9 \leq \lambda \leq 4.05$.
  \end{minipage}
\end{figure}

The graphs in Figure \ref{figure-1} and \ref{figure-2} indicate that for the linear system \eqref{equation-testequation1} it is favorable to choose 
the dilation value $\lambda$ in Algorithm \ref{algorithm-4} close to but smaller than the critical value $2\nu$. This fact can
be explained by the particular structure of the matrix $A$ and the right hand side $g_1^{\eps}$. The vector $g_1^{\eps}$ is up to a small perturbation exactly the eigenvector of 
$A^*A$ with respect to the smallest eigenvalue $0.0001$. For a fast termination
of Algorithm \ref{algorithm-4} it is thus favorable if the residual polynomials $\resa_n^{(\nu)*}$ decay fast at zero. This is guaranteed if the dilation parameter $\lambda$ is close to the 
critical value $2\nu$. For $\nu = 1$, the adaptive Algorithm \ref{algorithm-7} stops after $n = 95$ iterations and gives the optimal parameter $\lambda = 1.9930696$. 

Figures \ref{figure-1} and \ref{figure-2} also indicate that the number of iterations to solve \eqref{equation-testequation1}
is strongly linked to the smallest zero of the residual polynomials $\resa_n^{(\nu)*}$. The smallest zero of the residual polynomial $\resa_n^{(\nu)*}$ in $[0,1]$ is a decaying
function of the parameter $\lambda$ until the critical value $\lambda = 2 \nu$ is attained. At the critical value $\lambda = 2 \nu$ we cannot expect convergence of Algorithm \ref{algorithm-4} and this is
also verified in Figures \ref{figure-1} and \ref{figure-2}. For values of $\lambda$ larger than $2 \nu$ the smallest root of $\resa_n^{(\nu)*}$ is for large $n$ strictly less than zero. In this
case the second smallest root of $\resa_n^{(\nu)*}$ is the smallest root in the interval $[0,1]$. The convergence of Algorithm \ref{algorithm-4} is now linked to the position of the second smallest root of $\resa_n^{(\nu)*}$. 

After having considered a good-natured example, we give a second example in which Algorithm \ref{algorithm-4} does not improve if the parameter $\lambda$ is increased. We consider as a second test
equation
\begin{equation} \label{equation-testequation2} 
A f = g_2^{\eps}, \qquad g_2^{\eps} = e_2 + \eps w,
\end{equation}
with $A$ and $w$ given in \eqref{equation-testequation1} and $e_2 = (0,1,0, \cdots, 0)$. Again, we use Algorithm \ref{algorithm-4} with $\omega = 1$, $\eps = 0.01$, $N = 100$ to solve
\eqref{equation-testequation2} and stop the algorithm if the error $\|A f_n - g_2^{\eps}\|$ is less than $4\eps$.  
In this second case, the vector $g_2^{\eps}$ is up to a small perturbation the eigenvector of $A^*A$ with respect to the second largest eigenvalue $0.25$. 
Here, we can not expect that residual polynomials with a fast decay at the origin will have a strong effect on the number of iterations in Algorithm \ref{algorithm-4}.
The diagrams in Figure \ref{figure-3} confirm this expectation. The adaptive Algorithm \ref{algorithm-7} for $\nu = 1$ stops in this case after $n = 65$ iteration with the optimal parameter $\lambda = 1.6003658$.

\begin{figure}[H] 
\caption{Convergence of Algorithm \ref{algorithm-4} to solve \eqref{equation-testequation2} depending on the parameter $\lambda$.}
\label{figure-3} 
  \begin{minipage}{0.5\textwidth}
  \centering
  \includegraphics[angle=-90, width=\textwidth]{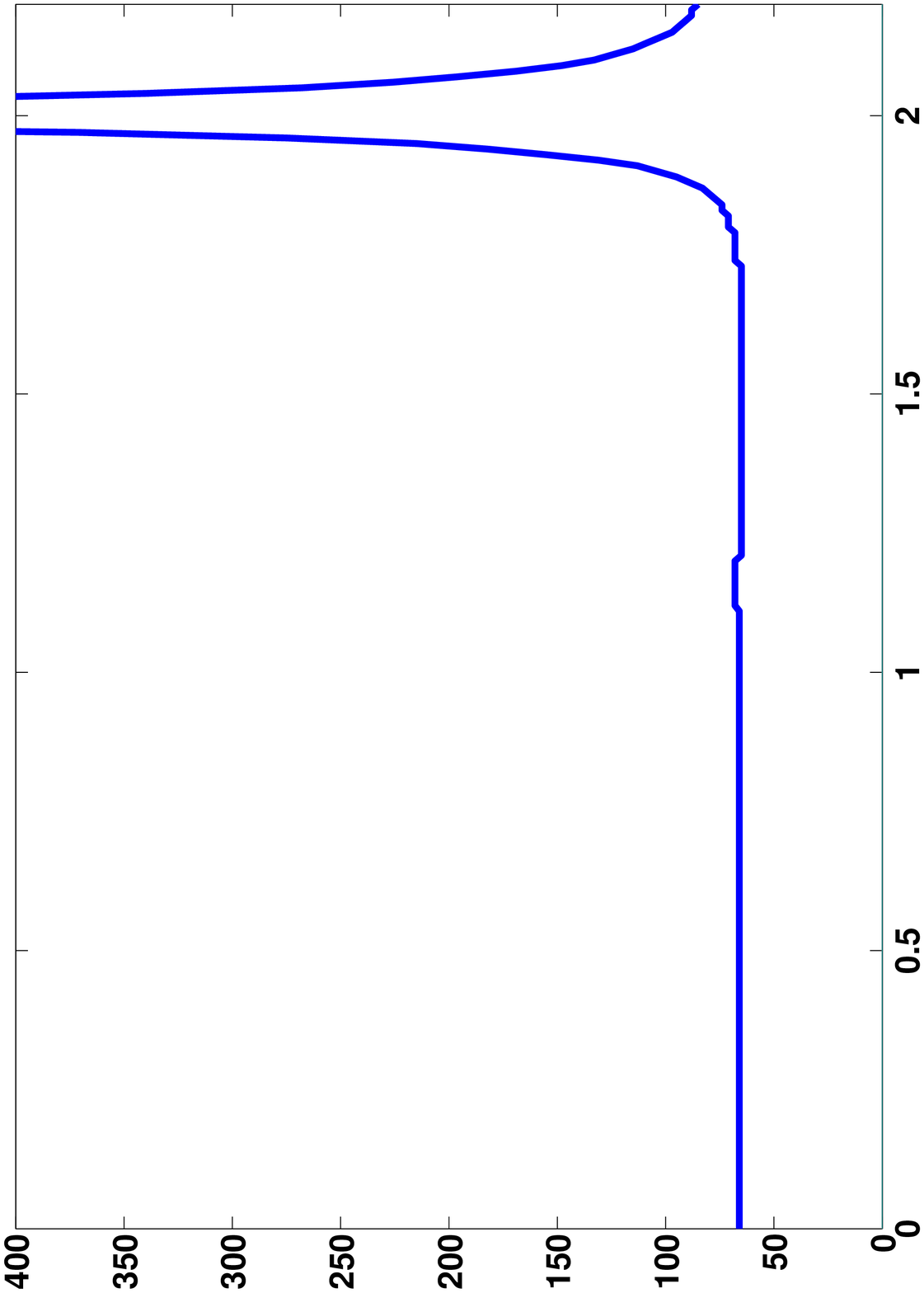}\\
  Number of iterations to solve \eqref{equation-testequation2} with Algorithm \ref{algorithm-4} with $\nu = 1$ depending on $0 \leq \lambda \leq 2.2$. 
  \end{minipage}\hfill
  \begin{minipage}{0.5\textwidth}
  \centering
  \includegraphics[angle=-90, width=\textwidth]{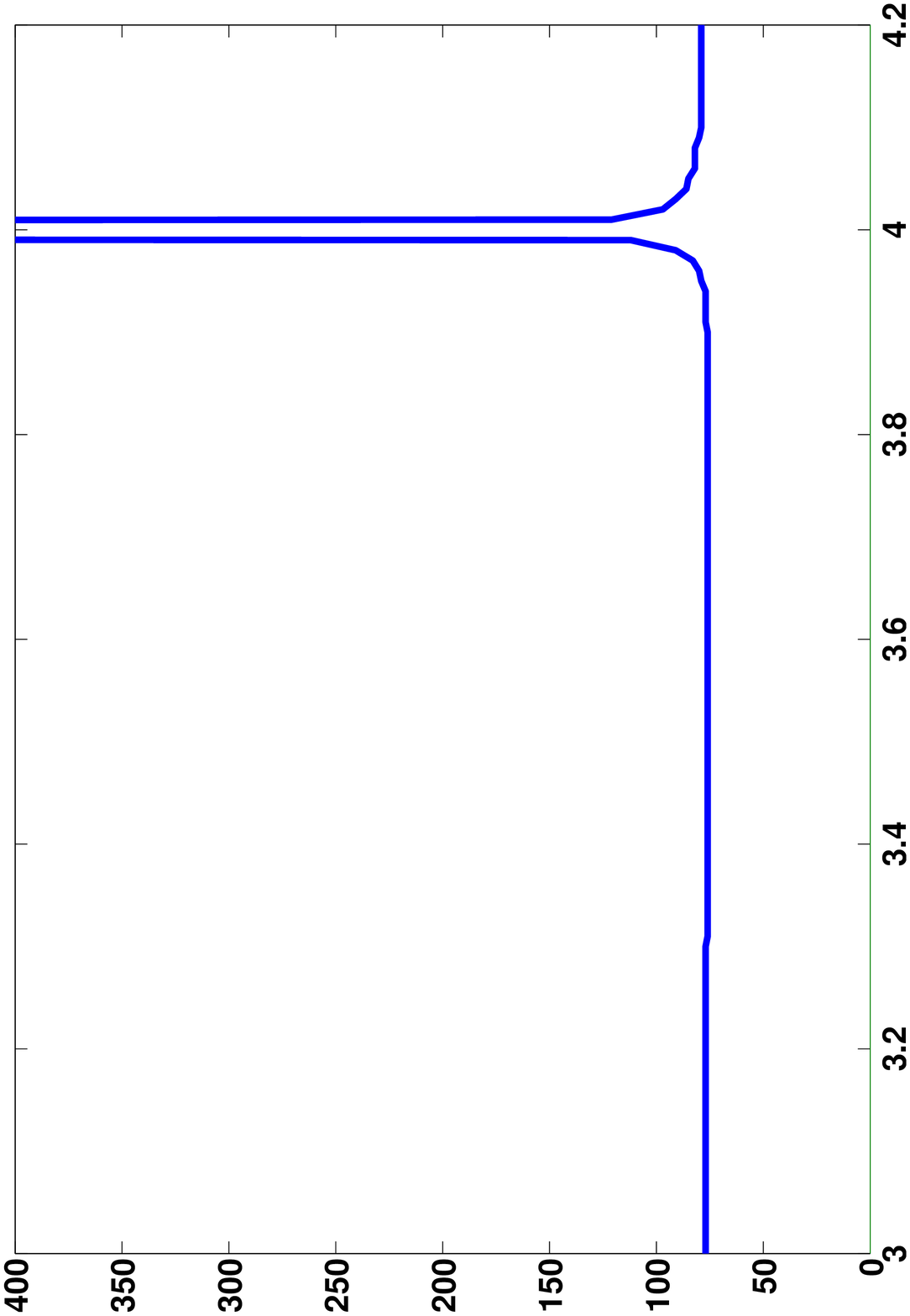}\\
  Number of iterations to solve \eqref{equation-testequation2} with Algorithm \ref{algorithm-4} with $\nu = 2$ depending on $3 \leq \lambda \leq 4.2$. 
  \end{minipage}
\end{figure}

As a final example of an ill-posed linear problem, we consider the well-known Fredholm integral equation of the first kind (see \cite[Example 12.4.1.]{DelvesMohamed}
\[ A f(s) = \int_0^1 k(s,t) f(t) dt = g(s),\qquad 0 \leq s \leq 1, \]
with the right hand side $g(s)$ and the kernel $k(s,t)$ given by
\[ g(s) = \frac{s^3-s}{6},\quad k(s,t) = \left\{ \begin{array}{ll} t(s-1) & 0\leq t < s \leq 1, \\ s(t-1) & 0\leq s \leq t \leq 1.
                                                             \end{array} \right. 
\]
The exact solution of this equation is given by the second derivative $g''(s) = s$. A discretized version of this integral equation based on the Galerkin method is included as test problem \texttt{deriv2} in the regularization toolbox of Hansen \cite{Hansen2007}. As a number of discretization points in \texttt{deriv2} we choose $N = 50$. Further, as in the previous test examples we disturb the right hand side $g$ by a vector $\eps w$ with $\eps = 0.01$. In Algorithm \ref{algorithm-4}, we choose $\omega = 96.5$ to guarantee $\omega \|A^* A\| < 1$ and stop the iteration according to the discrepancy principle if $\|A f_n - g^{\eps}\| < 4 \eps$ is satisfied. The different numbers of iterations in Algorithm \ref{algorithm-4} depending on the dilation parameter $\lambda$ are illustrated in Figure \ref{figure-5} and Table \ref{table-1}. 

\begin{figure}[H] 
\caption{Number of iteration steps of the co-dilated $\nu$-method to solve the test problem \texttt{deriv2} in the toolbox of Hansen \cite{Hansen2007}} \label{figure-5} 
  \begin{minipage}{0.5\textwidth}
  \centering
  \includegraphics[angle=-90, width=\textwidth]{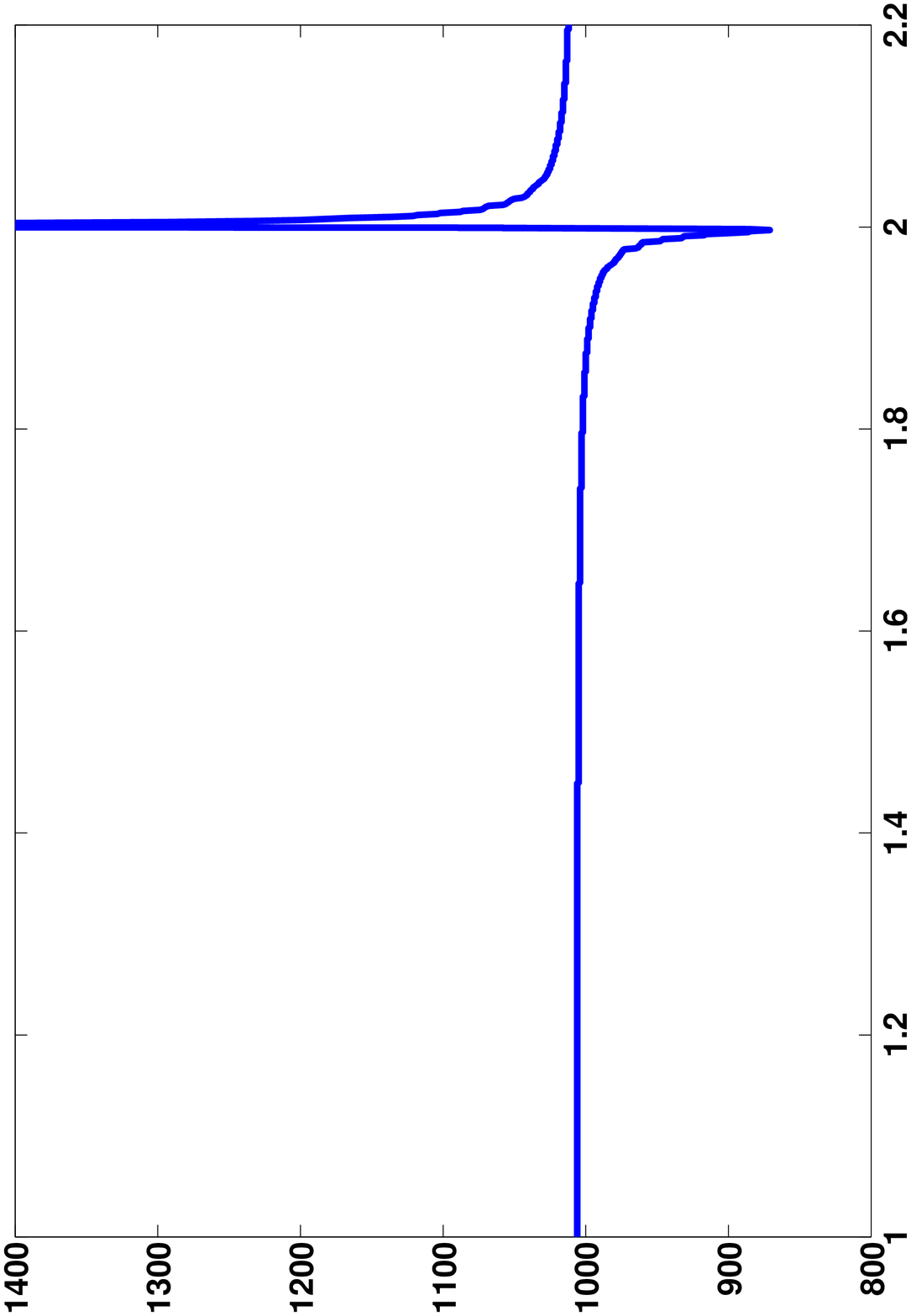}\\
  Number of iteration steps of the co-dilated $1$-method, depending on $1 \leq \lambda \leq 2.2$. 
  \end{minipage}\hfill
  \begin{minipage}{0.5\textwidth}
  \centering
  \includegraphics[angle=-90, width=\textwidth]{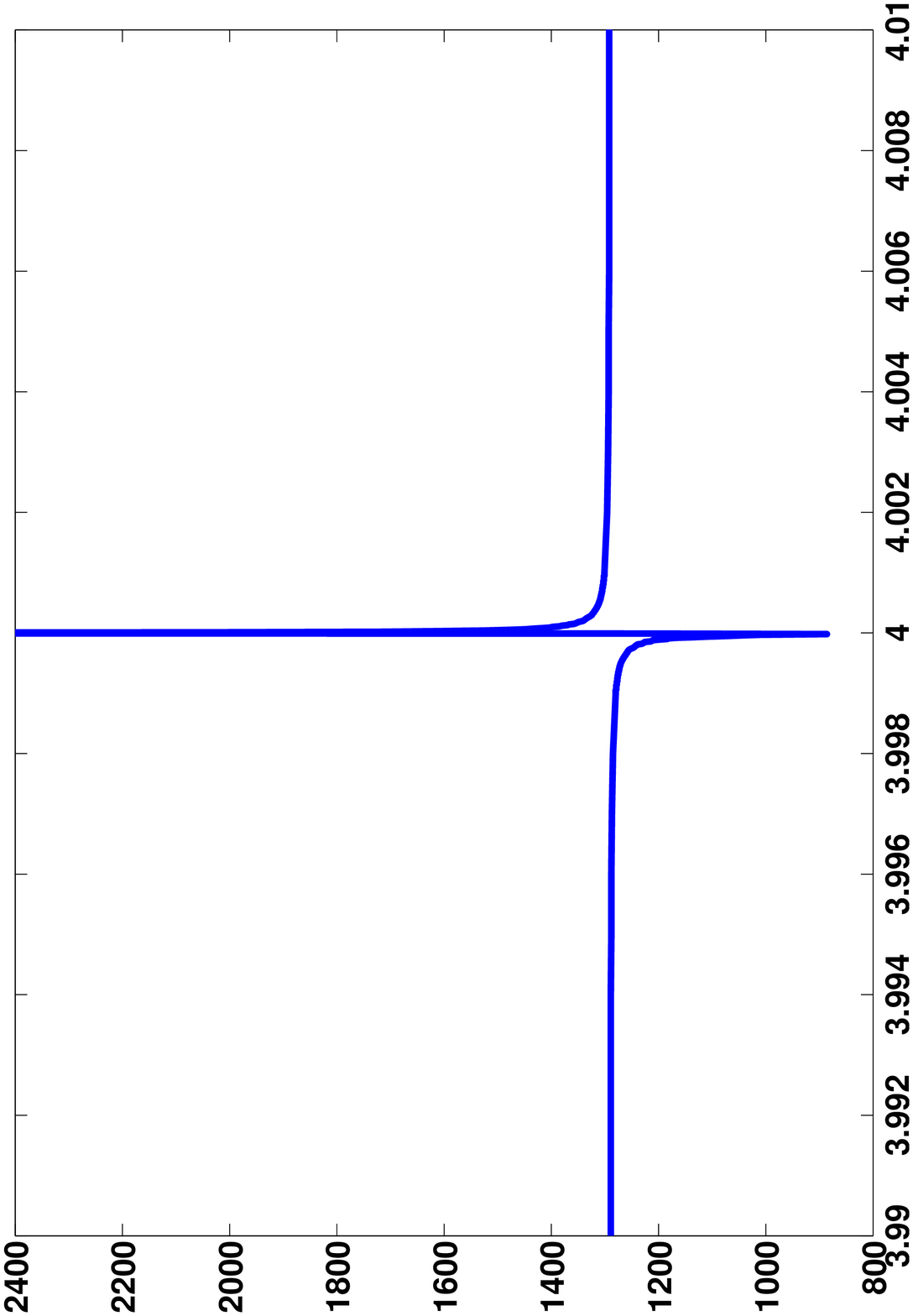}\\
  Number of iteration steps of the co-dilated $2$-method, depending on $3.99 \leq \lambda \leq 4.01$. 
  \end{minipage}
\end{figure}

\begin{table}[H] \caption{Convergence of different semi-iterative methods to solve the test problem \texttt{deriv2} 
}\label{table-1} 
\begin{center}
  \begin{tabular}[t]{|lc|lc|} \hline
    Semi-iterative method & Iteration steps &  Semi-iterative method & Iteration steps \\ \hline 
    $\nu = 1$, $\lambda = 0$ & $1007$ &  $\nu = 2$, $\lambda = 0$ & $1290$ \\
    $\nu = 1$, $\lambda = 0.5$ & $1007$ &  $\nu = 2$, $\lambda = 0.5$ & $1290$ \\
    \begin{minipage}{4cm}  $\nu = 1$, $\lambda = 1$ \\ \tiny $1$-method (Nemirovskii-Polyak) \end{minipage} & $1006$ & \begin{minipage}{4cm} $\nu = 2$, $\lambda = 1$ \\ \tiny $2$-method \end{minipage} & $1290$  \\[1mm]
    $\nu = 1$, $\lambda = 1.5$ & $1005$ &  $\nu = 2$, $\lambda = 3.9$ & $1290$ \\
    $\nu = 1$, $\lambda = 1.9$ & $998$ &  $\nu = 2$, $\lambda = 3.99$ & $1289$ \\
    $\nu = 1$, $\lambda = 1.99$ & $932$ &  $\nu = 2$, $\lambda = 3.999$ & $1280$ \\
    \begin{minipage}{4cm}  $\nu = 1$, $\lambda = 1.99716$ \\ \tiny Adaptive Algorithm \ref{algorithm-7} (optimal) \end{minipage} & $884$ &  $\nu = 2$, $\lambda = 3.9999$ & $1184$ \\
    $\nu = 1$, $\lambda = 1.9999$ & $1498$ &  $\nu = 2$, $\lambda = 3.99998$ & $886$ \\ \hline 
    cg-method & $23$ &  & \\ \hline 
    Landweber & $359379 $ &  & \\ \hline
    \end{tabular}
\end{center}
\end{table}

Figure \ref{figure-5} and Table \ref{table-1} illustrate that for the test problem \texttt{deriv2}, similar as for equation \eqref{equation-testequation1}, the total number of iteration steps of the co-dilated $\nu$-method gets significantly smaller if the parameter $\lambda$ approaches the critical value $2 \nu$. Also, if $\lambda$ is too close to $2 \nu$, we get very slow or no convergence of Algorithm \ref{algorithm-4}. For $\nu = 1$, the adaptive Algorithm \ref{algorithm-7} of the previous section gives the optimal parameter $\lambda = 1.99716$ after $n=884$ steps.

Further, Table \ref{table-1} shows that the $\nu$-methods and the co-dilated $\nu$-methods are significantly faster than the Landweber method. On the other hand, it is also visible that the cg-method outperforms
all semi-iterative methods in which the coefficients are a priori given. For a further comparison between the performance of the $\nu$-methods and the cg-iteration, we refer to \cite{Hanke1991}.

\section*{Acknowledgments}
I want to thank both referees very much for their excellent work. Their profound reviews and suggestions helped me a lot to improve this manuscript.

\end{document}